\pdfoutput=1
\documentclass[a4paper]{amsart}
\usepackage{CommandsAndStyle}

\usetikzlibrary{intersections}
\usetikzlibrary{fadings}
\usetikzlibrary{cd}

\DeclareMathOperator{\sal}{Sal}
\DeclareMathOperator{\Seq}{Seq}
\DeclareMathOperator{\Aut}{Aut}
\newcommand{\W}{\mathcal{W}}
\renewcommand{\C}{\mathcal{C}}
\newcommand{\Q}{\mathbb{Q}}
\renewcommand{\L}{\mathcal{L}}
\newcommand{\M}{\mathcal{M}}
\newcommand{\E}{\mathcal{E}}
\newcommand{\U}{\mathcal{U}}
\renewcommand{\H}{\mathcal{I}}
\DeclareMathOperator{\supp}{supp}

\DeclareMathOperator{\codim}{codim}
\newcommand{\cell}[1]{\langle #1 \rangle}
\DeclareMathOperator{\eu}{\dashv_{eu}}
\newcommand{\euprime}{\dashv'_{\text{eu}}}
\DeclareRobustCommand{\dashveq}{\text{\reflectbox{$\vDash$}\,}}
\definecolor{darkgreen}{RGB}{0,100,0}

\usepackage[foot]{amsaddr}

\title{Euclidean matchings and minimality of hyperplane arrangements}
\author{Davide Lofano}

\address{
	\textnormal{Scuola Normale Superiore, Piazza dei Cavalieri 7, 56126 Pisa (Italy)}}
\address{
	\textnormal{Current address (Davide Lofano): Technische Universität Berlin
(Germany)}
}
\email{lofano@math.tu-berlin.de}

\author{Giovanni Paolini}
\address{
	\textnormal{Current address (Giovanni Paolini): California Institute of Technology
(United States)}
}
\email{paolini@caltech.edu}

\thanks{\vskip0.01cm
The final publication in \emph{Discrete Mathematics} is available at \url{https://doi.org/10.1016/j.disc.2020.112232}.}

\begin{document}

\begin{abstract}
  We construct a new class of maximal acyclic matchings on the Salvetti complex of a locally finite hyperplane arrangement.
  Using discrete Morse theory, we then obtain an explicit proof of the minimality of the complement.
  Our construction provides interesting insights also in the well-studied case of finite arrangements, and gives a nice geometric description of the Betti numbers of the complement.
  In particular, we solve a conjecture of Drton and Klivans on the characteristic polynomial of finite reflection arrangements.
  The minimal complex is compatible with restrictions, and this allows us to prove the isomorphism of Brieskorn's Lemma by a simple bijection of the critical cells.
  Finally, in the case of line arrangements, we describe the algebraic Morse complex which computes the homology with coefficients in an abelian local system.
  
\end{abstract}

\maketitle

\section{Introduction}

Let $\A$ be a locally finite arrangement of affine hyperplanes in $\R^n$.
The complement $M(\A) \subseteq \mathbb{C}^n$ of the complexified arrangement $\A_\mathbb{C}$ is a well studied topological space.
As proved by Salvetti \cite{salvetti1987topology,salvetti1994homotopy}, $M(\A)$ has the homotopy type of an $n$-dimensional CW complex. This complex is usually called the \emph{Salvetti complex} of $\A$, and we denote it by $\sal(\A)$.

For a finite arrangement $\A$, in \cite{randell2002morse,dimca2003hypersurface,yoshinaga2007hyperplane} it was proved that the complement $M(\A)$ has the homotopy type of a minimal CW complex, i.e.\ with a number of $k$-cells equal to the $k$-th Betti number.
This minimality result was later made more explicit with discrete Morse theory, in \cite{salvetti2007combinatorial} (for finite affine arrangements), \cite{delucchi2008shelling} (for finite central arrangements and oriented matroids in general), \cite{gaiffi2009morse} (for finite line arrangements), \cite{dantonio2015minimality} (for affine arrangements with a finite number of directions).

In this work we consider a (possibly infinite) affine arrangement $\A$, and construct a minimal CW model for the complement $M(\A)$.
This is obtained applying discrete Morse theory to the Salvetti complex of $\A$.
For a (possibly infinite) CW complex, by ``minimal'' we mean that all the incidence numbers vanish.
As in the well known case of finite arrangements, we obtain a geometrically meaningful bijection between cells in the minimal CW model and chambers of $\A$.

Our starting point is the work of Delucchi on the minimality of oriented matroids \cite{delucchi2008shelling}.
Specifically, we build on the idea of decomposing the Salvetti complex according to some ``good'' total order of the chambers.
For a general affine arrangement, however, the combinatorial order used in \cite{delucchi2008shelling} does not yield a decomposition with the desired properties.
In Section \ref{sec:decomposition} we introduce a class of total orders of the chambers for which we are able to extend the construction of Delucchi, and we call them \emph{valid orders}.
We remark that in \cite[Question 4.18]{delucchi2008shelling} it was explicitly asked for one such extension to affine arrangements.
For a finite affine arrangement, the polar order of Salvetti and Settepanella \cite{salvetti2007combinatorial} is valid (Remark \ref{rmk:polar-order}).
Therefore our work contributes to linking the constructions of \cite{salvetti2007combinatorial} and \cite{delucchi2008shelling} (see also \cite[Remark 3.8]{delucchi2008shelling}).

In Section \ref{sec:matching} we show how to construct an acyclic matching on $\sal(\A)$ for any given valid order.

\begin{customthm}{\ref{thm:matching}}
  Let $\A$ be a locally finite hyperplane arrangement, with a given valid order of the set of chambers.
  Then there exists a proper acyclic matching on $\sal(\A)$ with critical cells in bijection with the chambers.
\end{customthm}

In the same section we also prove the following result that can be regarded as a generalization of \cite[Theorem 3.6]{delucchi2008shelling}.

\begin{customthm}{\ref{teo:fiber}}
  Let $X$ be a $k$-dimensional polytope in $\R^k$, and let $y \in \R^k$ be a point outside $X$ that does not lie in the affine hull of any facet of $X$.
  Then there exists an acyclic matching on the poset of faces of $X$ visible from $y$, such that no face is critical.
\end{customthm}

In Section \ref{sec:euclidean} we construct valid orders for any locally finite arrangement $\A$, considering the Euclidean distance of the chambers from a fixed generic point $x_0 \in \R^n$.
In this way, we obtain a family of matchings on $\sal(\A)$ that we call \emph{Euclidean matchings}.
The idea of constructing a minimal complex that depends on a ``generic point'' appears to be new, as opposed to the more classical approach of using a ``generic flag'' \cite{yoshinaga2007hyperplane, salvetti2007combinatorial, gaiffi2009morse}.
The critical cells are in bijection with the chambers, and can be described explicitly.

\begin{customthm}{\ref{thm:euclidean-matching}}
  Let $\A$ be a locally finite arrangement in $\R^n$.
  For every generic point $x_0 \in \R^n$, there exists a Euclidean matching on $\sal(\A)$ with base point $x_0$.
  Such a matching has exactly one critical cell $\cell{C, F_C}$ for every chamber $C \in \C(\A)$,
  where $F_C$ is the smallest face of $C$ that contains the projection of $x_0$ onto $C$.
\end{customthm}

We prove that the Morse complex of a Euclidean matching is minimal.

\begin{customthm}{\ref{thm:minimality}}
  Let $\A$ be a locally finite hyperplane arrangement in $\R^n$, and let $\M$ be a Euclidean matching on $\sal(\A)$ with base point $x_0$.
  Then the associated Morse complex $\sal(\A)_{\M}$ is minimal (i.e.\ all the incidence numbers vanish).
\end{customthm}

In particular, we obtain a new geometric way to read the Betti numbers and the Poincaré polynomial of $M(\A)$ from the arrangement $\A$.
This solves a conjecture of Drton and Klivans on the coefficients of the characteristic polynomial of a finite reflection arrangement \cite{drton2010geometric}.

\begin{customcor}{\ref{cor:betti-numbers}}
  Let $\A$ be a (locally) finite hyperplane arrangement in $\R^n$, and let $x_0 \in \R^n$ be a generic point.
  The $k$-th Betti number of the complement $M(\A)$ is equal to the number of chambers $C$ such that the projection of $x_0$ onto $C$ lies in the relative interior of a face $F_C$ of codimension $k$.
  Equivalently, the Poincaré polynomial of $\A$ is given by
  \[ \pi(\A,t) = \sum_{C \in \,\C(\A)} t^{\,\codim F_C}. \]
\end{customcor}

In Section \ref{sec:brieskorn} we use Euclidean matchings to obtain a proof of Brieskorn's Lemma (for locally finite complexified arrangements) which makes no use of algebraic geometry.
In addition, we show that for every flat $X$ there exist Euclidean matchings on $\sal(\A)$ for which the Morse complex of the subarrangement $\A_X$ is naturally included into the Morse complex of $\A$.

Finally, in Section \ref{sec:homology} we give an explicit description of the algebraic Morse complex that computes the homology of $M(\A)$ with coefficients in an abelian local system, for any locally finite line arrangement $\A$ in $\R^2$.
We compare our result with the one of Gaiffi and Salvetti \cite{gaiffi2009morse}, where similar formulas are obtained in the case of finite line arrangements (using the polar matchings of Salvetti and Settepanella \cite{salvetti2007combinatorial}).
 \section{Background and notations}

In this section we briefly recall some basic definitions and results about hyperplane arrangements, discrete Morse theory, polytopes, and shellability.

\subsection{Hyperplane arrangements}
\label{sec:hyperplane-arrangements}

See \cite{orlik2013arrangements} for a general reference about hyperplane arrangements.
Our notations mostly agree with those of \cite{salvetti2007combinatorial} and \cite{delucchi2008shelling}.

Let $\A$ be a locally finite arrangement of affine hyperplanes in $\R^n$.
Denote by $M(\A) \subseteq \mathbb{C}^n$ the complement of the complexified arrangement $\A_\mathbb{C}$.

The arrangement $\A$ gives rise to a stratification of $\R^n$ into topological subspaces called \emph{faces} (see \cite[Chapter 5]{bourbaki1968elements}).
It is more convenient for us to work with the closure of these subspaces, so we assume from now on that the faces are closed.
By \emph{relative interior} of a face $F$ we mean the topological interior of $F$ inside the affine span of $F$.
The faces of codimension $0$ are called \emph{chambers}.
Denote the set of faces by $\F = \F(\A)$, and the set of the chambers by $\C = \C(\A)$.
The set $\F$ has a natural partial order: $F \preceq G$ if and only if $F \supseteq G$.
The poset $\F$ is called the \emph{face poset} of $\A$, and it is ranked by codimension.

Given two chambers $C,C' \in \C$, let $s(C,C') \subseteq \A$ be the set of hyperplanes which separate $C$ and $C'$.
Also, denote by $\W_C \subseteq \A$ the set of hyperplanes that intersect $C$ in a face of codimension $1$.
These hyperplanes are called \emph{walls} of $C$.

For every chamber $C$, the set $\C$ can be endowed with a partial order $\leq_C$ defined as follows: $D' \leq_C D$ if and only if $s(C,D') \subseteq s(C,D)$.
In the language of oriented matroids, $(\C, \leq_C)$ is called the \emph{tope poset based at $C$} \cite[Definition 4.2.9]{bjorner1999oriented}.

Let $\L = \L(\A)$ be the poset of intersections of the hyperplanes in $\A$, ordered by reverse inclusion.
An element $X \in \L$ is called a \emph{flat}.
Notice that the entire space $\R^n$ is an element of $\L$ (being the intersection of zero hyperplanes), and it is in fact the unique minimal element of $\L$.
The poset $\L$ is a geometric semilattice called the \emph{poset of flats}, and it is also ranked by codimension.
Denote by $\L_k(\A) \subseteq \L(\A)$ the set of flats of codimension $k$.

For a subset $U \subseteq \R^n$ (usually a face or a flat), let $\supp(U) \subseteq \A$ be the subarrangement of $\A$ consisting of the hyperplanes that contain $U$.
This is called the \emph{support} of $U$.
Also, denote by $|U| \subseteq \R^n$ the affine span of $U$.
Notice that, for a face $F \in \F$, we have $|F| \in \L$.

Given a flat $X \in \L$, we also use the notation $\A_X$ to indicate the support of $X$ (this operation is called \emph{restriction}).
Denote by $\A^X$ the arrangement in $X$ given by $\{ H \cap X \mid H \not\in \A_X \}$ (this operation is called \emph{contraction}).
Let $\pi_X\colon \C(\A) \to \C(\A_X)$ be the natural projection, which maps a chamber $C \in \C(\A)$ to the unique chamber of $\A_X$ that contains $C$.

For a chamber $C \in \C$ and a face $F \in \F$, denote by $C.F$ the unique chamber $C' \preceq F$ such that $\pi_{|F|}(C) = \pi_{|F|}(C')$.
In other words, this is the unique chamber containing $F$ and lying in the same chamber as $C$ in $\A_{|F|}$.
In addition, if $C \preceq F$, denote by $C^F$ the chamber opposite to $C$ with respect to $F$.

The Salvetti complex of $\A$, first introduced in \cite{salvetti1987topology}, is a regular CW complex homotopy equivalent to the complement $M(\A)$ in $\mathbb{C}^n$ (see also \cite{gelfand1989algebraic, bjorner1992combinatorial, salvetti1994homotopy, orlik2013arrangements}).
Its poset of cells $\sal(\A)$ is defined as follows.
There is a $k$-cell $\cell{C,F}$ for each pair $(C,F)$ where $C \in \C$ is a chamber and $F \in \F$ is a face of $C$ of codimension $k$.
A cell $\cell{C,F}$ is in the boundary of $\cell{D,G}$ if and only if $F \prec G$ and $D.F = C$.

\begin{theorem}[\cite{salvetti1987topology}]
  The poset $\sal(\A)$ is the poset of cells of a regular CW complex homotopy equivalent to $M(\A)$.
\end{theorem}

\subsection{Discrete Morse theory}
We recall here the main concepts of Forman's discrete Morse theory \cite{forman1998morse,forman2002user}.
We follow the point of view of Chari \cite{chari2000discrete}, using acyclic matchings instead of discrete Morse functions, and we make use of the generality of \cite[Section 3]{batzies2002discrete} for the case of infinite CW complexes.

Let $(P,<)$ be a ranked poset.
If $q < p$ in $P$ and there is no element $r\in P$ with $q<r<p$, then we write $q \lessdot p$.
Given $p \in P$ we define $P_{\leq p}= \{q \in P \mid q \leq p\}$.

Let $G$ be the Hasse diagram of $P$, i.e.\ the graph with vertex set $P$ and having an edge $(p, q)$ whenever $q \lessdot p$.
Denote by $\E = \{ (p,q) \in P\times P \mid q \lessdot p \}$ the set of edges of $G$.

Given a subset $\M$ of $\E$, we can orient all edges of $G$ in the following way: an edge $(p,q) \in \E$ is oriented from $p$ to $q$ if the pair does not belong to $\M$, otherwise in the opposite direction.
Denote this oriented graph by $G_{\M}$.

\begin{definition}[Acyclic matching \cite{chari2000discrete}]
  A \emph{matching} on $P$ is a subset $\M \subseteq \E$ such that every element of $P$ appears in at most one edge of $\M$.
  A matching $\M$ is \emph{acyclic} if the graph $G_{\M}$ has no directed cycle.
\end{definition}

Given a matching $\M$ on $P$, an \emph{alternating path} is a directed path in $G_{\M}$ such that two consecutive edges of the path do not both belong to $\M$ or both to $\E \setminus \M$.
The elements of $P$ that do not appear in any edge of $\M$ are called \emph{critical} (with respect to the matching $\M$).

\begin{definition}[Grading \cite{batzies2002discrete}]
  Let $Q$ be a poset.
  A poset map $\varphi \colon P \to Q$ is called a \emph{$Q$-grading} of $P$.
  The $Q$-grading $\varphi$ is \emph{compact} if $\varphi^{-1}(Q_{\leq q}) \subseteq P$ is finite for all $q \in Q$.
  A matching $\M$ on $P$ is \emph{homogeneous} with respect to the $Q$-grading $\varphi$ if $\varphi(p) = \varphi(p')$ for all $(p,p') \in \M$.
  An acyclic matching $\M$ is \emph{proper} if it is homogeneous with respect to some compact grading.
\end{definition}

The following is a direct consequence of the definition of a proper matching (cf.\ \cite[Definition 3.2.5 and Remark 3.2.17]{batzies2002discrete}).

\begin{lemma}[\cite{batzies2002discrete}]\label{lemma:alternating}
  Let $\M$ be a proper acyclic matching on a poset $P$, and let $p \in P$.
  Then there is a finite number of alternating paths starting from $p$, and each of them has a finite length.
\end{lemma}

We are ready to state the main theorem of discrete Morse theory. This particular formulation follows from \cite[Theorem 3.2.14 and Remark 3.2.17]{batzies2002discrete}

\begin{theorem}[\cite{forman1998morse,chari2000discrete,batzies2002discrete}]
  Let $X$ be a regular CW complex, and let $P$ be its poset of cells.
  If $\M$ is a proper acyclic matching on $P$, then $X$ is homotopy equivalent to a CW complex $X_\M$ (called the \emph{Morse complex} of $\M$) with cells in dimension-preserving bijection with the critical cells of $X$.
\end{theorem}

The construction of the Morse complex is explicit in terms of the CW complex $X$ and the matching $\M$ (see for example \cite{batzies2002discrete}).
This allows us to obtain relations between the incidence numbers with $\Z$ coefficients in the Morse complex and incidence numbers in the starting complex.

\begin{theorem}[{\cite[Theorem 3.4.2]{batzies2002discrete}}]
  \label{thm:incidence}
  Let $X$ be a regular CW complex, $P$ its poset of cells and $\M$ a proper acyclic matching on $P$.
  Let $X_\M$ be the Morse complex of $\M$.
  Given two critical cells $\sigma, \tau \in X$ with $\dim \sigma = \dim \tau +1$, denote by $\sigma_\M$ and $\tau_\M$ the corresponding cells in $X_\M$.
  Then the incidence number between $\sigma_\M$ and $\tau_\M$ in $X_\M$ is given by
  \begin{equation*}
    [\sigma_\M : \tau_\M]_{X_\M} \,= \sum_{\gamma \in \Gamma(\sigma,\tau)}{m(\gamma)},
  \end{equation*}
  where $\Gamma(\sigma,\tau)$ is the set of all alternating paths between $\sigma$ and $\tau$.
  If $\gamma \in \Gamma(\sigma,\tau)$ is of the form
  \begin{equation*}
    \sigma=\sigma_0 \searrow \tau_1 \nearrow \sigma_1 \searrow \ldots \tau_k \nearrow \sigma_k \searrow \tau,
  \end{equation*}
  then $m(\gamma)$ is given by
  \begin{equation*}
    m(\gamma) \,=\, (-1)^{k}[\sigma_k : \tau] \, \prod_{i=1}^k{[\sigma_{i-1}:\tau_i][\sigma_i:\tau_i]}.
  \end{equation*}
\end{theorem}

Finally, recall the following standard tool for constructing acyclic matchings.

\begin{theorem}[Patchwork theorem {\cite[Theorem 11.10]{kozlov2007combinatorial}}]\label{teo:patchwork}
  Let $\varphi \colon P \to Q$ be a $Q$-grading of $P$.
  For all $q \in Q$, assume to have an acyclic matching $\M_q \subseteq \E$ that involves only elements of the subposet $\varphi^{-1}(q) \subseteq P$.
  Then the union of these matchings is itself an acyclic matching on $P$.
\end{theorem}

\subsection{Polyhedra, polytopes, and shellability}
\label{sec:shellability}

In this section we briefly recall some notions and results from \cite{ziegler2012lectures}.

\begin{definition}
  A \emph{polyhedron} is an intersection of finitely many closed halfspaces in some $\R^d$.
  A \emph{polytope} is a bounded polyhedron.
\end{definition}

Given a polyhedron $P$, denote by $\F(P)$ the complex of its faces (considering the polyhedron $P$ itself as a trivial face).
The faces of codimension $1$ are called \emph{facets}.
In addition, denote by $\F(\partial P)$ the \emph{boundary complex} of $P$, i.e.\ the complex that contains only the proper faces of $P$.

 \begin{definition}
  We say that a facet $G \in \F(P)$ is \emph{visible} from a point $p \in \R^d$ if every line segment from $p$ to a point of $G$ does not intersect the interior of $P$ (cf.\ \cite[Theorem 8.12]{ziegler2012lectures}).
  We say that a face $F \in \F(P)$ is \emph{visible} from $p$ if all the facets $G \supseteq F$ of $P$ are visible from $p$.
  In particular, notice that the entire polyhedron $P$ is always visible from $p$.
\end{definition}

We are now able to recall the notion of shellability of the boundary complex of a polytope.

\begin{definition}[{\cite[Definition 8.1]{ziegler2012lectures}}]
  A \emph{shelling} of the boundary complex of a polytope $P$ is a linear ordering $F_1,F_2,\ldots,F_s$ of the facets of $P$ such that either the facets are points, or the following conditions are satisfied.
  
  \begin{enumerate}
    \item The boundary complex $\F(\partial F_1)$ of the first facet has a shelling.
  
    \item For $1<j \leq s$, the intersection of the facet $F_j$ with the previous facets is nonempty and is a beginning segment of a shelling of $\F(\partial F_j)$, that is
    \begin{equation*}
      F_j \cap \left(\bigcup_{i=1}^{j-1}{F_i} \right)=G_1 \cup G_2 \cup\dots \cup G_r
    \end{equation*}
    for some shelling $G_1,G_2,\ldots,G_r,\ldots,G_t$ of $F_j$, with $1 \leq r \leq t$.
    A facet $F_j$ is called a \emph{spanning facet} if $r=t$.
  \end{enumerate}
  A polytope is \emph{shellable} if its boundary complex has a shelling.
\end{definition}

To conclude, recall the following two results about shellability of the boundary complex of a polytope.

\begin{lemma}[{\cite[Lemma 8.10]{ziegler2012lectures}}]
  \label{lemma:reverse-shelling}
  If $F_1,F_2,\ldots,F_s$ is a shelling order for the boundary of a polytope $P$, then so is the reverse order $F_s,F_{s-1},\ldots,F_1$.
\end{lemma}

\begin{theorem}[\cite{bruggesser1972shellable}, {\cite[Theorem 8.12]{ziegler2012lectures}}]
  \label{thm:line-shelling}
  Let $P \subseteq \R^d$ be a $d$-polytope, and let $x \in \R^d$ be a point outside $P$.
  If $x$ lies in general position (that is, not in the affine hull of a facet of $P$), then the boundary complex of the polytope has a shelling in which the facets of $P$ that are visible from $x$ come first.
\end{theorem}

 \section{Decomposition of the Salvetti complex}
\label{sec:decomposition}

Our aim is to construct an acyclic matching on the Salvetti complex of a locally finite affine arrangement $\A$, with critical cells in explicit bijection with the chambers of $\A$.
Following the ideas of Delucchi \cite{delucchi2008shelling}, we want to decompose the Salvetti complex into ``pieces'' (one piece for every chamber) and construct an acyclic matching on each of these pieces with exactly one critical cell.
More formally, we are going to decompose the poset of cells $\sal(\A)$ as a disjoint union
\[ \sal(\A) = \bigsqcup_{C \in \C} N(C), \]
so that every subposet $N(C) \subseteq \sal(\A)$ admits an acyclic matching with one critical cell.

\begin{definition}
  Given a chamber $C \in \C$, let $S(C) \subseteq \sal(\A)$ be the set of all the cells $\cell{C',F} \in \sal(\A)$ such that $C' = C.F$.
In other words, a cell $\langle C' , F \rangle $ is in $S(C)$ if all the hyperplanes in $\supp(F)$ do not separate $C$ and $C'$.
\end{definition}

Notice that the cells in $S(C)$ form a subcomplex of the Salvetti complex (using poset terminology, $S(C)$ is a lower ideal in $\sal(\A)$).
This subcomplex is dual to the stratification of $\R^n$ induced by $\A$.
Also, the natural map $S(C) \to \F$ which sends $\cell{C',F}$ to $F$ is a poset isomorphism.

Now fix a total order $\dashv$ of the chambers:
\[ \C = \{ C_0 \dashv C_1 \dashv C_2 \dashv \dots \} \]
(when $\C$ is infinite, the order type is that of natural numbers).

\begin{definition}
  For every chamber $C \in \C$, let $N(C) \subseteq S(C)$ be the subset consisting of all the cells not included in any $S(C')$ with $C' \dashv C$.
\end{definition}

The union of the subcomplexes $S(C)$, for $C \in \C$, is the entire complex $\sal(\A)$.
Thus the subsets $N(C)$, for $C \in \C$, form a partition of $\sal(\A)$.
All the $0$-cells are contained in $N(C_0) = S(C_0)$.
Therefore, for $C \neq C_0$, the cells of $N(C)$ do not form a subcomplex of the Salvetti complex.
If $\A$ is a (finite) central arrangement, this definition of $N(C)$ coincides with the one given in \cite[Section 4]{delucchi2008shelling}.

We want now to choose the total order $\dashv$ of the chambers so that each $N(C)$ admits an acyclic matching with exactly one critical cell.
In \cite{delucchi2008shelling}, this is done taking any linear extension of the partial order $\leq_{C_0}$, for any base chamber $C_0$.
Such a total order works well for central arrangements but not for general affine arrangements, as we see in the following two examples.

\begin{example}
  Consider a non-central arrangement of three lines in the plane, as in Figure \ref{fig:two-arrangements} on the left.
  Choose $C_0$ to be one of the three unbounded chambers with two walls.
  In any linear extension of $\leq_{C_0}$, the last chamber $C_6$ must be the non-simplicial unbounded chamber opposite to $C_0$.
  However, $S(C_6) \subseteq \bigcup_{C \neq C_6} S(C)$, so $N(C_6)$ is empty, and therefore it does not admit an acyclic matching with one critical cell.
  Figure \ref{fig:non-valid} shows the decomposition of the Salvetti complex for one of the possible linear extensions of $\leq_{C_0}$.
  \label{example:non-valid}\end{example}

\begin{example}
  Consider the arrangement of five lines depicted on the right in Figure \ref{fig:two-arrangements}.
  For every choice of a base chamber $C_0$ and for every linear extension of $\leq_{C_0}$, there is some chamber $C$ such that $N(C)$ is empty.
\end{example}

\newcommand{\arrangement}{
  \draw (3.5,0) -- (3.5,4);
  \draw (0,1.25) -- (5,3.5);
  \draw (0,2.75) -- (5,0.5);
}
\newcommand{\chambers}{
  \coordinate (C0) at (0.5,2) {};
  \coordinate (C1) at (2,1) {};
  \coordinate (C2) at (2,3) {};
  \coordinate (C3) at (2.9,2) {};
  \coordinate (C4) at (4.1,0.4) {};
  \coordinate (C5) at (4.1,3.6) {};
  \coordinate (C6) at (4.3,2) {};
}

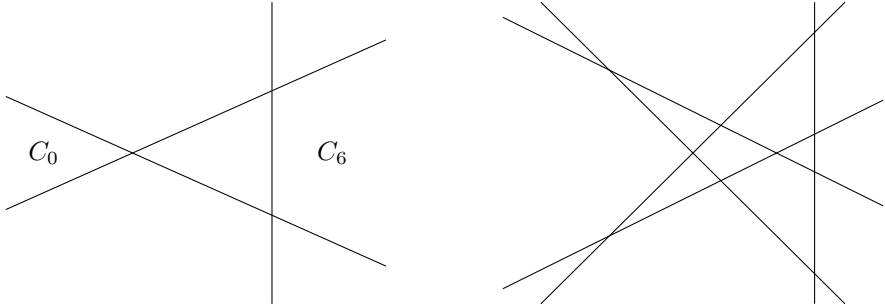
\begin{figure}
  \begin{tikzpicture}
    \arrangement
    \node (C0) at (0.5,2) {$C_0$};
    \node (C1) at (2,1) {};
    \node (C2) at (2,3) {};
    \node (C3) at (2.9,2) {};
    \node (C4) at (4.1,0.4) {};
    \node (C5) at (4.1,3.6) {};
    \node (C6) at (4.3,2) {$C_6$};
  \end{tikzpicture}
  \qquad\qquad
  \begin{tikzpicture}
    \draw (4.1,0) -- (4.1,4);
    \draw (0,0.2) -- (5,2.7);
    \draw (0,3.8) -- (5,1.3);
    \draw (0.5,4) -- (4.5,0);
    \draw (0.5,0) -- (4.5,4);
  \end{tikzpicture}
  
  \caption{Two line arrangements.}
  \label{fig:two-arrangements}
\end{figure}

We are now going to state a condition on the total order $\dashv$ on $\C$ that produces a decomposition of the Salvetti complex with the desired properties.
First recall the following definition from \cite{delucchi2008shelling}.

\begin{definition}
  Given a chamber $C$ and a total order $\dashv$ on $\C$, let
  \[ \J(C) = \{ X \in \L \mid \supp(X) \cap s(C,C')\neq \emptyset \;\; \forall\, C' \dashv C \}. \]
\end{definition}

Notice that $\J(C)$ is an upper ideal of $\L$, and it coincides with $\L$ for $C=C_0$.
In \cite[Theorem 4.15]{delucchi2008shelling} it is proved that, if $\A$ is a (finite) central arrangement and $\dashv$ is a linear extension of $\leq_{C_0}$ (for any choice of $C_0 \in \C$), then $\J(C)$ is a principal upper ideal for every chamber $C \in \C$.
This is the condition we need.

\begin{definition}[Valid order]
  A total order $\dashv$ on $\C$ is \emph{valid} if, for every chamber $C \in \C$, $\J(C)$ is a principal upper ideal generated by some flat $X_C = |F_C| \in \L$ where $F_C$ is a face of $C$.
  \label{def:valid-order}\end{definition}

The total orders of Example \ref{example:non-valid} are not valid, because $\J(C_6)$ is empty.
A valid order that begins with the chamber $C_0$ of Example \ref{example:non-valid} is shown in Figure \ref{fig:valid}.

The previous definition is the starting point of our answer to \cite[Question 4.18]{delucchi2008shelling}, where it was asked for an extension of the arguments of \cite{delucchi2008shelling} to affine arrangements.
Sections \ref{sec:matching} and \ref{sec:euclidean} will motivate this definition.

\begin{remark}
  \label{rmk:polar-order}
  If $\A$ is a finite affine arrangement, the polar order of the chambers defined by Salvetti and Settepanella \cite[Definition 4.5]{salvetti2007combinatorial} is valid.
  Indeed, $\J(C)$ is a principal upper ideal generated by $X_C = |F_C|$, where $F_C$ is the smallest face of $C$ with respect to the polar order of the faces.
  Therefore Definition \ref{def:valid-order} highlights the link between the constructions of \cite{salvetti2007combinatorial} and \cite{delucchi2008shelling} (see also \cite[Remark 3.8]{delucchi2008shelling}).
  The results of Section \ref{sec:matching}, if applied to polar orders, give rise to acyclic matchings that are related to the polar matchings of \cite{salvetti2007combinatorial}.
\end{remark}

\tikzstyle{salvettivertices}=[every node/.style={circle,inner sep=2pt,fill=black}]
\tikzstyle{salvettiedges}=[every path/.style={->,very thick}]
\tikzstyle{XC}=[every path/.style={red, dashed, very thick}, every node/.style={red}]

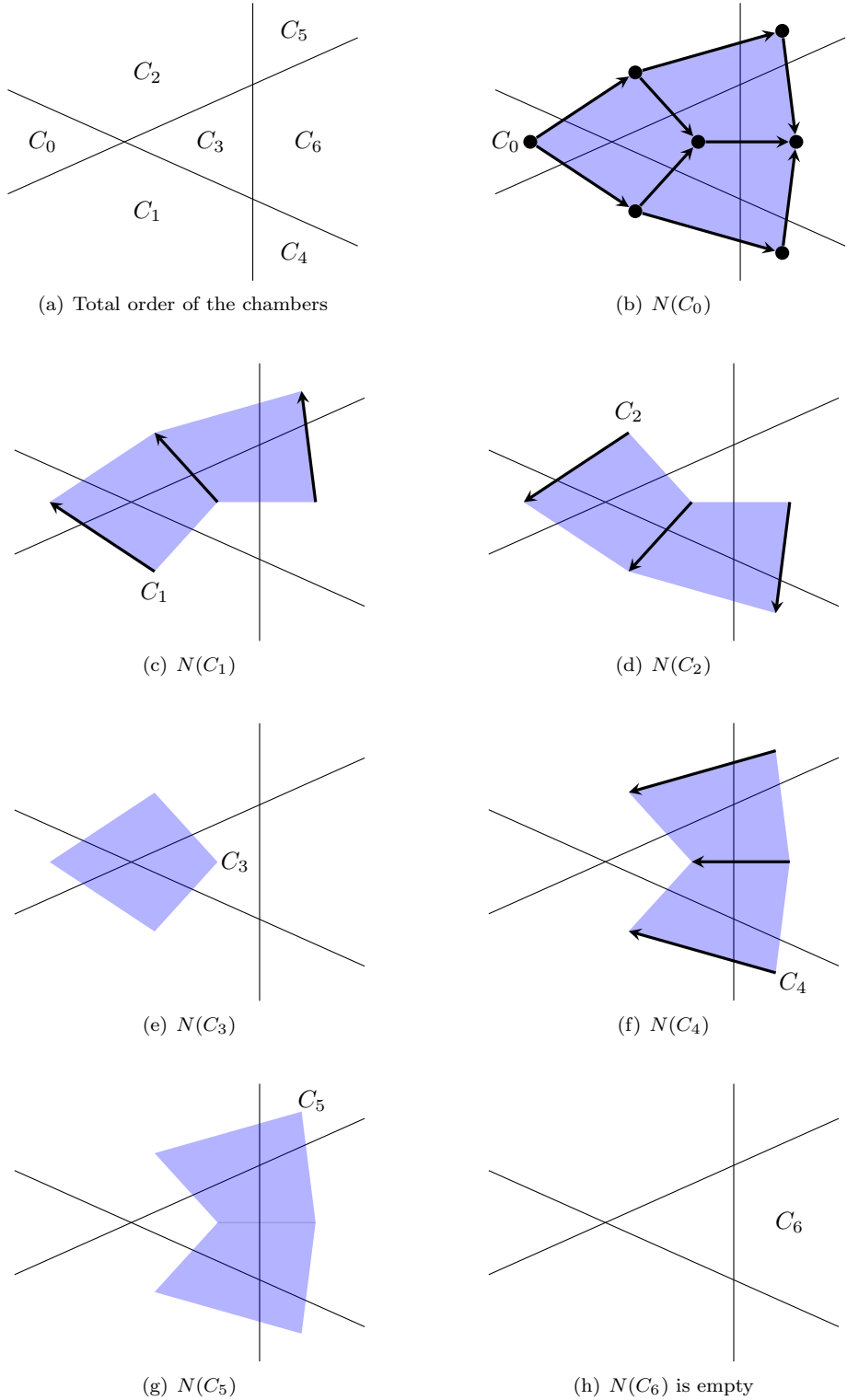
\begin{figure}
  \subfigure[Total order of the chambers]{
  \begin{tikzpicture}
    \arrangement
    \node (C0) at (0.5,2) {$C_0$};
    \node (C1) at (2,1) {$C_1$};
    \node (C2) at (2,3) {$C_2$};
    \node (C3) at (2.9,2) {$C_3$};
    \node (C4) at (4.1,0.4) {$C_4$};
    \node (C5) at (4.1,3.6) {$C_5$};
    \node (C6) at (4.3,2) {$C_6$};
  \end{tikzpicture}
  }
  \qquad\qquad
  \subfigure[$N(C_0)$]{
    \begin{tikzpicture}[salvettivertices]
      \arrangement
      \chambers
       \node[circle,inner sep=2pt,fill=white] (c0) at (0.15,2) {$C_0$};
      \begin{scope}[every path/.style={opacity=0.3, left color=blue, right color=white}]
        \fill (C0) -- (C1) -- (C3) -- (C2);
        \fill (C2) -- (C3) -- (C6) -- (C5);
        \fill (C1) -- (C3) -- (C6) -- (C4);
      \end{scope}
      \foreach \i in {0,...,6}
        \node (P\i) at (C\i) {};
      \begin{scope}[salvettiedges]
        \draw (P0) -> (P1);
        \draw (P0) -> (P2);
        \draw (P1) -> (P3);
        \draw (P2) -> (P3);
        \draw (P1) -> (P4);
        \draw (P2) -> (P5);
        \draw (P3) -> (P6);
        \draw (P4) -> (P6);
        \draw (P5) -> (P6);
      \end{scope}
     
    \end{tikzpicture}
  }
  
  \bigskip
  
  \subfigure[$N(C_1)$]{
    \begin{tikzpicture}
      \arrangement
      \chambers
      \begin{scope}[every path/.style={opacity=0.3, bottom color=blue, top color=white}]
        \fill (C0) -- (C1) -- (C3) -- (C2);
        \fill[shading angle=-20] (C2) -- (C3) -- (C6) -- (C5);
      \end{scope}
      \begin{scope}[salvettiedges]
        \draw (C1) -> (C0);
        \draw (C3) -> (C2);
        \draw (C6) -> (C5);
      \end{scope}
      \node (C1) at (2,0.7) {$C_1$};
    \end{tikzpicture}
  }
  \qquad\qquad
  \subfigure[$N(C_2)$]{
    \begin{tikzpicture}
      \arrangement
      \chambers
      \begin{scope}[every path/.style={opacity=0.3, top color=blue, bottom color=white}]
        \fill (C0) -- (C1) -- (C3) -- (C2);
        \fill[shading angle=20] (C3) -- (C1) -- (C4) -- (C6);
      \end{scope}
      \begin{scope}[salvettiedges]
        \draw (C2) -> (C0);
        \draw (C3) -> (C1);
        \draw (C6) -> (C4);
      \end{scope}
      \node (C2) at (2,3.3) {$C_2$};
    \end{tikzpicture}
  }
  
  \bigskip
  
  \subfigure[$N(C_3)$]{
    \begin{tikzpicture}
      \arrangement
      \chambers
      \begin{scope}[every path/.style={opacity=0.3, right color=blue, left color=white}]
        \fill (C0) -- (C1) -- (C3) -- (C2);
      \end{scope}
      \node (C3) at (3.15,2) {$C_3$};
    \end{tikzpicture}
  }
  \qquad\qquad
  \subfigure[$N(C_4)$]{
    \begin{tikzpicture}
      \arrangement
      \chambers
      \begin{scope}[every path/.style={opacity=0.3, bottom color=blue, top color=white}]
        \fill[shading angle=50] (C5) -- (C2) -- (C3) -- (C6);
        \fill[shading angle=20] (C3) -- (C1) -- (C4) -- (C6);
      \end{scope}
      \begin{scope}[salvettiedges]
        \draw (C4) -> (C1);
        \draw (C6) -> (C3);
        \draw (C5) -> (C2);
      \end{scope}
      \node (C4) at (4.35,0.25) {$C_4$};
    \end{tikzpicture}
  }
  
  \bigskip
  
  \subfigure[$N(C_5)$]{
    \begin{tikzpicture}
      \arrangement
      \chambers
      \begin{scope}[every path/.style={opacity=0.3, top color=blue, bottom color=white}]
        \fill[shading angle=-40] (C5) -- (C2) -- (C3) -- (C6);
        \fill[shading angle=-50] (C3) -- (C1) -- (C4) -- (C6);
      \end{scope}
      \node (C5) at (4.25,3.75) {$C_5$};
    \end{tikzpicture}
  }
  \qquad\qquad
  \subfigure[$N(C_6)$ is empty]{
    \begin{tikzpicture}
      \arrangement
      \chambers
      \node (C6) at (4.3,2) {$C_6$};
    \end{tikzpicture}
  }
  
  \caption{A non-central arrangement of three lines in the plane, with a linear extension of $\leq_{C_0}$.
  Here $N(C_5)$ and $N(C_6)$ do not admit acyclic matchings with one critical cell.}
  \label{fig:non-valid}
\end{figure}

\begin{figure}
  \subfigure[Total order of the chambers]{
  \begin{tikzpicture}
    \arrangement
    \node (C0) at (0.5,2) {$C_0$};
    \node (C1) at (2,1) {$C_1$};
    \node (C2) at (2,3) {$C_2$};
    \node (C3) at (2.9,2) {$C_3$};
    \node (C4) at (4.1,0.4) {$C_5$};
    \node (C5) at (4.1,3.6) {$C_6$};
    \node (C6) at (4.3,2) {$C_4$};
  \end{tikzpicture}
  }
  \qquad\qquad
  \subfigure[$N(C_0)$]{
    \begin{tikzpicture}[salvettivertices]
      \arrangement
      \chambers
        \node[circle,inner sep=2pt,fill=white] (c0) at (0.15,2) {$C_0$};
      \begin{scope}[every path/.style={opacity=0.3, left color=blue, right color=white}]
        \fill (C0) -- (C1) -- (C3) -- (C2);
        \fill (C2) -- (C3) -- (C6) -- (C5);
        \fill (C1) -- (C3) -- (C6) -- (C4);
      \end{scope}
      \foreach \i in {0,...,6}
        \node (P\i) at (C\i) {};
      \begin{scope}[salvettiedges]
        \draw (P0) -> (P1);
        \draw (P0) -> (P2);
        \draw (P1) -> (P3);
        \draw (P2) -> (P3);
        \draw (P1) -> (P4);
        \draw (P2) -> (P5);
        \draw (P3) -> (P6);
        \draw (P4) -> (P6);
        \draw (P5) -> (P6);
      \end{scope}
    \end{tikzpicture}
  }
  
  \bigskip
  
  \subfigure[$N(C_1)$]{
    \begin{tikzpicture}
      \arrangement
      \chambers
      \begin{scope}[every path/.style={opacity=0.3, bottom color=blue, top color=white}]
        \fill (C0) -- (C1) -- (C3) -- (C2);
        \fill[shading angle=-20] (C2) -- (C3) -- (C6) -- (C5);
      \end{scope}
      \begin{scope}[salvettiedges]
        \draw (C1) -> (C0);
        \draw (C3) -> (C2);
        \draw (C6) -> (C5);
      \end{scope}
      \begin{scope}[XC]
        \draw (0,1.25) -- (5,3.5);
        \node at (0.5,1) {$X_{C_1}$};
      \end{scope}
       \node (C1) at (2,0.7) {$C_1$};
    \end{tikzpicture}
  }
  \qquad\qquad
  \subfigure[$N(C_2)$]{
    \begin{tikzpicture}
      \arrangement
      \chambers
      \begin{scope}[every path/.style={opacity=0.3, top color=blue, bottom color=white}]
        \fill (C0) -- (C1) -- (C3) -- (C2);
        \fill[shading angle=20] (C3) -- (C1) -- (C4) -- (C6);
      \end{scope}
      \begin{scope}[salvettiedges]
        \draw (C2) -> (C0);
        \draw (C3) -> (C1);
        \draw (C6) -> (C4);
      \end{scope}
      \begin{scope}[XC]
        \draw (0,2.75) -- (5,0.5);
        \node at (0.5,2.9) {$X_{C_2}$};
      \end{scope}
      \node (C2) at (2,3.3) {$C_2$};
    \end{tikzpicture}
  }
  
  \bigskip
  
  \subfigure[$N(C_3)$]{
    \begin{tikzpicture}
      \arrangement
      \chambers
      \begin{scope}[every path/.style={opacity=0.3, right color=blue, left color=white}]
        \fill (C0) -- (C1) -- (C3) -- (C2);
      \end{scope}
      \filldraw[red] (1.667,2) circle (2pt);
      \begin{scope}[XC]
        \node at (1.6,2.4) {$X_{C_3}$};
      \end{scope}
      \node (C3) at (3.15,2) {$C_3$};
    \end{tikzpicture}
  }
  \qquad\qquad
  \subfigure[$N(C_4)$]{
    \begin{tikzpicture}
      \arrangement
      \chambers
      \begin{scope}[every path/.style={opacity=0.3, right color=blue, left color=white}]
        \fill[shading angle=53] (C5) -- (C2) -- (C3) -- (C6);
        \fill[shading angle=130] (C3) -- (C1) -- (C4) -- (C6);
      \end{scope}
      \begin{scope}[salvettiedges]
        \draw (C4) -> (C1);
        \draw (C6) -> (C3);
        \draw (C5) -> (C2);
      \end{scope}
      \begin{scope}[XC]
        \draw (3.5,0) -- (3.5,4);
        \node at (3.1,3.7) {$X_{C_4}$};
      \end{scope}
      \node (C4) at (4.6,2) {$C_4$};
    \end{tikzpicture}
  }
  
  \bigskip
  
  \subfigure[$N(C_5)$]{
    \begin{tikzpicture}
      \arrangement
      \chambers
      \begin{scope}[every path/.style={opacity=0.3, bottom color=blue, top color=white}]
        \fill[shading angle=55] (C3) -- (C1) -- (C4) -- (C6);
      \end{scope}
      \filldraw[red] (3.5,1.175) circle (2pt);
      \begin{scope}[XC]
        \node at (3,1.1) {$X_{C_5}$};
      \end{scope}
      \node (C5) at (4.35,0.25) {$C_5$};
    \end{tikzpicture}
  }
  \qquad\qquad
  \subfigure[$N(C_6)$]{
    \begin{tikzpicture}
      \arrangement
      \chambers
      \begin{scope}[every path/.style={opacity=0.3, top color=blue, bottom color=white}]
        \fill[shading angle=-52] (C5) -- (C2) -- (C3) -- (C6);
      \end{scope}
      \filldraw[red] (3.5,2.825) circle (2pt);
      \begin{scope}[XC]
        \node at (3,2.95) {$X_{C_6}$};
      \end{scope}
      \node (C6) at (4.3,3.75) {$C_6$};
    \end{tikzpicture}
  }
  
  \caption{A non-central arrangement of three lines in the plane, with a valid order of the chambers.
For every chamber $C$ except $C_0$, the generator $X_C$ of $\J(C)$ is highlighted.}
  \label{fig:valid}
\end{figure}
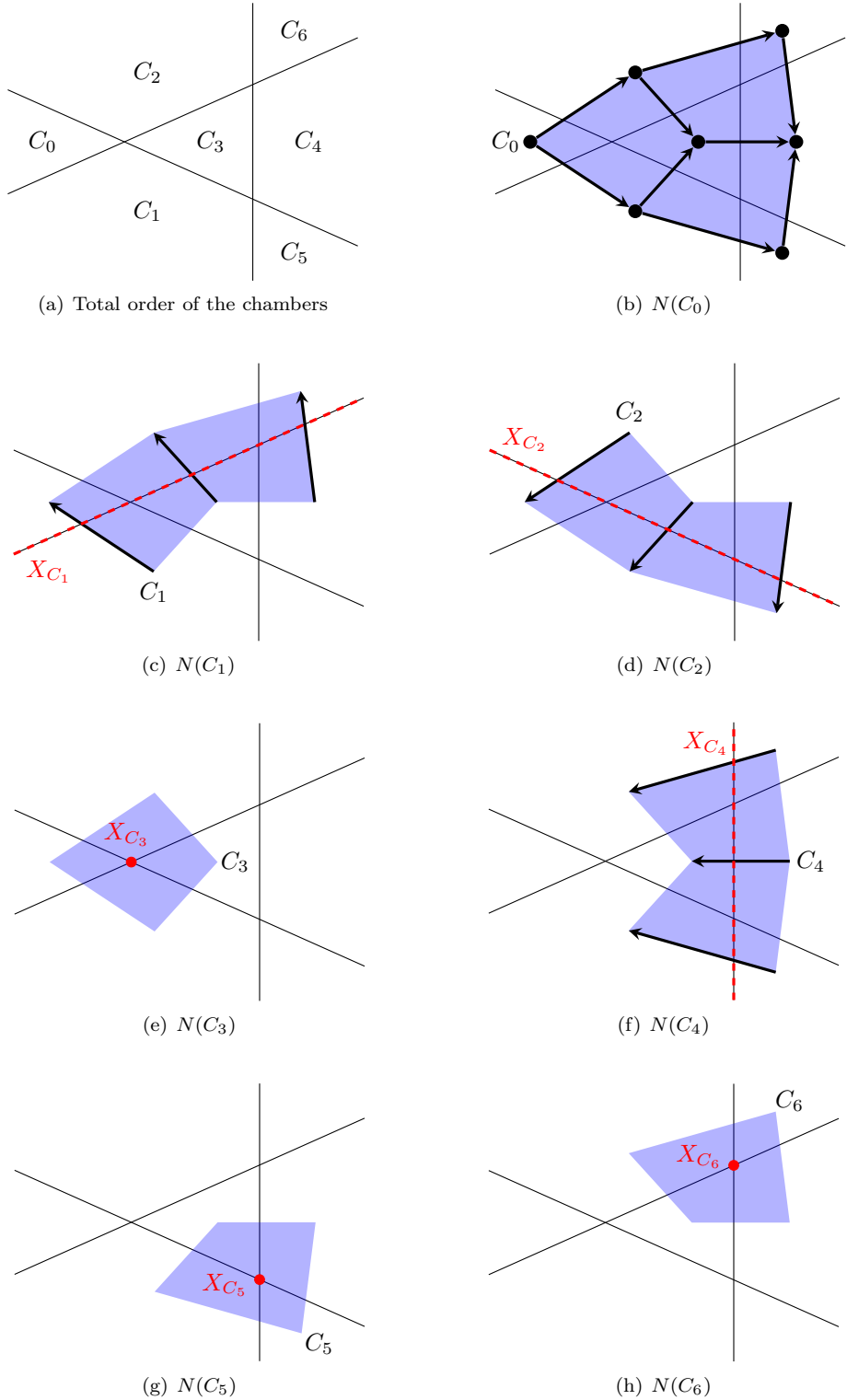
  \section{Construction of the acyclic matching}
\label{sec:matching}

Throughout this section we assume that we have an arrangement $\A$ together with a valid order $\dashv$ of $\C$ (as in Definition \ref{def:valid-order}).
Using the decomposition
\[ \sal(\A) = \bigsqcup_{C \in \C} N(C) \]
of Section \ref{sec:decomposition} (induced by the valid order $\dashv$), we are going to construct a proper acyclic matching on $\sal(A)$ with critical cells in bijection with the chambers.
More precisely, we are going to construct an acyclic matching on every $N(C)$ with exactly one critical cell, and then attach these matchings together using the Patchwork Theorem (Theorem \ref{teo:patchwork}).
This strategy is the same as the one employed in \cite{delucchi2008shelling}, but our proofs are different since we deal with affine and possibly infinite arrangements.

\begin{lemma}
  \label{lemma:NC}
  Suppose that $\dashv$ is a valid order of $\C$, in the sense of Definition \ref{def:valid-order}. Then 
  \begin{equation*}
    N(C) = \{ \cell{D,F} \in S(C) \mid F \subseteq X_C\}.
  \end{equation*}
\end{lemma}

\begin{proof}
  To prove the inclusion $\subseteq$, assume by contradiction that there exists some cell $\cell{D,F} \in N(C)$ with $F \nsubseteq X_C$.
  By minimality of $X_C$ in $\J(C)$, we have that $|F| \notin \J(C)$.
  This means that there exists a chamber $C' \dashv C$ such that $\supp(F) \cap s(C,C') = \emptyset$. Then $C$ and $C'$ are contained in the same chamber of $\A_{|F|}$, which implies $C'.F=C.F$.
  By definition of $S(C)$, we have that $C.F=D$.
  Then $C'.F=D$, so $\cell{D,F} \in S(C')$.
  This is a contradiction, since $\cell{D,F} \in N(C)$ and $C' \dashv C$.
  
  For the opposite inclusion, consider a cell $\cell{D,F} \in S(C)$ with $F \subseteq X_C$.
  Then $|F| \in \J(C)$, so for every chamber $C' \dashv C$ there exists an hyperplane in $\supp(F) \cap s(C,C')$.
  By the same argument as before we can deduce that $D=C.F \neq C'.F$ for all $C' \dashv C$, which means that $\cell{D,F} \notin S(C')$ for all $C' \dashv C$.
  Therefore $\cell{D,F} \in N(C)$.
\end{proof}

Recall that, for a chamber $D \in \C$ and a face $F \succeq D$, we denote by $D^F$ the chamber opposite to $D$ with respect to $F$.
For every chamber $C \in \C$, consider the map
\begin{equation*}
  \tilde \eta_C \colon S(C) \to \mathcal{C}
\end{equation*}
that sends a cell $\cell{D,F}$ to $D^F$.

\begin{lemma}
  The map $\tilde \eta_C\colon S(C) \to (\C, \leq_C)$ is order-preserving.
  \label{lemma:eta-order-preserving}
\end{lemma}

\begin{proof}
  Let $\cell{D,F}, \cell{D',F'} \in S(C)$, and suppose that $\cell{D',F'} \leq \cell{D,F}$ (see Figure \ref{Fig:lemma4.2}).
  Then $F' \preceq F$ and therefore $\supp(F') \subseteq \supp(F)$.
  Call $E=D^F$ and $E'=D'^{F'}$.
  By definition of $S(C)$, we have that $s(C,E)=s(C,D) \cup \supp(F)$ and $s(C,E')=s(C,D') \cup \supp(F')$.
  In addition, $F' \preceq F$ implies that $s(D,D') \subseteq \supp(F) \setminus \supp(F')$.
  Since $s(C,D') \subseteq s(C,D) \cup s(D,D')$, we conclude that
  \begin{align*}
  s(C,E') &= s(C,D') \cup \supp(F') \subseteq s(C,D) \cup s(D,D') \cup \supp(F') \\
   &\subseteq s(C,D) \cup \supp(F) = s(C,E).
  \end{align*}
  Therefore $E' \leq_C E$.
\end{proof}

\begin{figure}
 
  \begin{tikzpicture}
    \draw (0,0) -- (10,5);
    \draw (0,4.5) -- (10,2);
    \draw (5,5) -- (7.5,0);
    \draw (2,0) -- (3,5);
    
   \node[circle,inner sep=2pt,fill=black] (f) at (6,3) {};
   \node (F) at (5.9,2.65) {$F$};
   \node (C) at (1,2.5) {$C$};
   \node (D) at (3.7,2.7) {$D$};
   \node (E) at (8.5,3.2) {$E$};
   \node[text=blue] (D') at (4,4.5) {$D'$};
   \node[text=blue] (F') at (5.8,4.1) {$F'$};
   \node[text=blue] (E') at (7,4.5) {$E'$};
   
   \draw[draw=blue, thick] (5,5) -- (6,3);
  \end{tikzpicture}
  
  \caption{Proof of Lemma \ref{lemma:eta-order-preserving}.}
  \label{Fig:lemma4.2}
\end{figure}
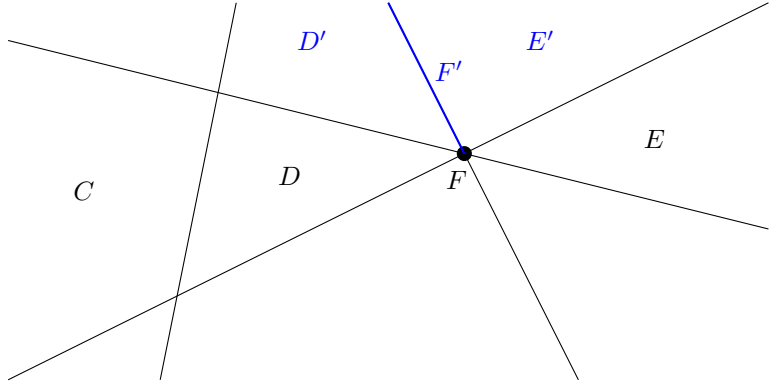

Consider the restriction $\eta_C = \tilde \eta_C |_{N(C)} \colon N(C) \to \C$.
The matching on $N(C)$ will be obtained as a union of acyclic matchings on each fiber $\eta_C^{-1}(E)$ of $\eta_C$.
Lemma \ref{lemma:eta-order-preserving}, together with the Patchwork Theorem, will ensure that the matching on $N(C)$ is acyclic.
We now fix two chambers $C$ and $E$, and study the fiber $\eta_{C}^{-1}(E)$.

\begin{lemma}
  Let $\dashv$ be a valid order of $\C$, and let $C,E$ be two chambers.
  A cell $\cell{D, F} \in \sal(\A)$ is in the fiber $\eta_C^{-1}(E)$ if and only if $D=E^F$, $F \subseteq X_C$, and $\supp(F) \subseteq s(C,E)$.
  \label{lemma:fiber}
\end{lemma}

\begin{proof}
  Suppose that $\cell{D,F} \in \eta_C^{-1}(E)$.
  In particular, $\cell{D,F} \in N(C)$, thus by Lemma \ref{lemma:NC} we have that $F \subseteq X_C$.
  By definition of $\eta_C$, $D^F=E$ and so $E^F=D$.
  Finally, we have $\supp(F) \subseteq s(D,E)$ by definition of $\eta_C$, and $\supp(F) \cap s(C,D) = \emptyset$ by definition of $S(C)$, so $\supp(F) \subseteq s(D,E) \setminus s(C,D) \subseteq s(C,E)$.
  
  We want now to prove that a cell $\cell{D,F}$ that satisfies the given conditions is in the fiber $\eta_C^{-1}(E)$.
  Since $D$ is opposite to $E$ with respect to $F$, we deduce that $\supp(F) \subseteq s(D,E)$.
  Then, using the hypothesis $\supp(F) \subseteq s(C,E)$, we obtain $\supp(F) \cap s(C,D) = \emptyset$.
  This means that $C.F = D$, i.e.\ $\cell{D,F} \in S(C)$.
  By Lemma \ref{lemma:NC}, we conclude that $\cell{D,F} \in N(C)$.
  The fact that $\eta_C(\cell{D,F}) = E$ follows directly from the definition of $\eta_C$.
\end{proof}

A cell $\cell{D,F}$ in the fiber $\eta_C^{-1}(E)$ is determined by $F$, because $D=E^F$.
Thus we immediately have the following corollary.

\begin{provedcorollary}\label{corollary:visible-unrestricted}
  The fiber $\eta_{C}^{-1}(E)$ is in order-preserving (and rank-preserving) bijection with the set of faces $F \succeq E$ such that $F \subseteq X_C$ and $\supp(F) \subseteq s(C,E)$.
  In particular, if $\eta_C^{-1}(E)$ is non-empty, then $\supp(X_C) \subseteq s(C,E)$.
\end{provedcorollary}

Assume from now on that the fiber $\eta_C^{-1}(E)$ is non-empty.
The above corollary can be restated as follows, restricting to the flat $X_C$.

\begin{corollary}\label{corollary:visible}
  Suppose that the fiber $\eta_C^{-1}(E)$ is non-empty.
  Then $C' = C \cap X_C$ and $E' = E \cap X_C$ are chambers of the (contraction) arrangement $\A^{X_C}$, and $\eta_{C}^{-1}(E)$ is in order-preserving bijection with the set of faces $F \succeq E'$ such that $\supp(F) \subseteq s(C',E')$ in $\A^{X_C}$.
\end{corollary}

\begin{proof}
  By Definition \ref{def:valid-order}, $X_C = |F_C|$ for some face $F_C$ of $C$. Then $C' = C \cap X_C = F_C$ is a chamber of $\A^{X_C}$.
  
  Consider now any cell $\cell{D,F} \in \eta_{C}^{-1}(E)$, and let $D'=D \cap X_C$.
  If we prove that $D'$ is a chamber of $\A^{X_C}$, then the same is true for $E'$, since they are opposite with respect to $F$ and $F \subseteq X_C$ (by Lemma \ref{lemma:NC}).
  Let $F'_C=F_C.F$ in the arrangement $\A^{X_C}$ (so $F'_C$ is a chamber of $\A^{X_C}$), and consider the chamber $\tilde D = C.F'_C$ in $\A$.
  Then $\tilde D = C.F = D$ (the first equality holds because $F'_C \preceq F$, and the second equality because $D \in S(C)$).
  Therefore $D' = D \cap X_C = \tilde D \cap X_C = F'_C$ is a chamber of $\A^{X_C}$.
  
  The second part is mostly a rewriting of Corollary \ref{corollary:visible-unrestricted}, but some care should be taken since we are passing from the arrangement $\A$ to the arrangement $\A^{X_C}$.
  To avoid confusion, in $\A^{X_C}$ write $\supp'$ and $s'$ in place of $\supp$ and $s$.
  Given a face $F \subseteq X_C$, we need to prove that $\supp(F) \subseteq s(C,E)$ in $\A$ if and only if $\supp'(F) \subseteq s'(C',E')$ in $\A^{X_C}$.
  This is true because
  \begin{align*}
    \supp'(F) &= \{ H \cap X_C \mid H \in \supp(F) \text{ and } H \nsupseteq X_C \}; \\
    s'(C',E') &= \{ H \cap X_C \mid H \in s(C,E) \text{ and } H \nsupseteq X_C \}. \qedhere
  \end{align*}
\end{proof}

Constructing an acyclic matching on $\eta_C^{-1}(E)$ is then the same as constructing an acyclic matching on the set of faces of $E'$ given by Corollary \ref{corollary:visible}.
We start by considering the special case $E'=C'$.

\begin{lemma}
  \label{lemma:special-fiber}
  Suppose that the fiber $\eta_C^{-1}(E)$ is non-empty.
  Then $E' = C'$ if and only if $E$ is the chamber opposite to $C$ with respect to $X_C$.
  In this case, $\eta_C^{-1}(E)$ contains the single cell $\cell{C,F_C}$.
\end{lemma}

\begin{proof}
  If $E$ is opposite to $C$ with respect to $X_C$, then clearly $E' = C'$.
  Conversely, suppose that $E' = C' = F_C$.
  Let $\cell{D,F}$ be any cell in $\eta_C^{-1}(E)$.
  As in the proof of Corollary \ref{corollary:visible}, we have that $D \cap X_C = F_C'$, where $F_C' = F_C.F$ in $\A^{X_C}$.
  Notice that $F \subseteq E \cap X_C = E' = F_C$, so $F_C' = F_C.F = F_C$.
  In other words, the chambers $C$, $D$ and $E$ all contain the face $F_C$.
  Since $F \subseteq F_C \subseteq C \cap D$, we have that $s(C,D) \subseteq \supp(F)$.
  But $D \in S(C)$ implies that $D = C.F$, i.e.\ $s(C,D) \cap \supp(F) = \emptyset$.
  Therefore $s(C,D) = \emptyset$, so $C=D$.
  Now, $E$ is the opposite of $D$ with respect to $F$, and $E \cap X_C = D \cap X_C = F_C$, so $F = F_C$.
  This means that $E$ is the opposite of $C$ with respect to $X_C$.
  The previous argument also shows that $\eta_C^{-1}(E)$ contains the single cell $\cell{C,F_C}$.
\end{proof}

In particular, for every chamber $C$ there is exactly one fiber $\eta_C^{-1}(E)$ for which $E'=C'$.
This fiber contains exactly one cell, which is going to be critical with respect to our matching.

Consider now the case $E' \neq C'$.
In view of Corollary \ref{corollary:visible}, we work with the restricted arrangement $\A^{X_C}$ in $X_C$.
Until Lemma \ref{lemma:unb-bounded}, all our notations (for example, $\supp(F)$ and $s(C',E')$) are intended with respect to the arrangement $\A^{X_C}$.
In what follows, we make use of the definitions and facts of Section \ref{sec:shellability}.

\begin{lemma}
  \label{lemma:visible-faces}
  Let $y_{C'}$ be a point in the interior of $C'$.
  The faces $F \succeq E'$ such that $\supp(F) \subseteq s(C',E')$ are exactly the faces of $E'$ that are visible from $y_{C'}$.
\end{lemma}

\begin{proof}
  Suppose that $\supp(F) \subseteq s(C',E')$.
  In particular, for every facet $G \supseteq F$ of $E'$, the hyperplane $|G| \in \A^{X_C}$ separates $C'$ and $E'$ and so $G$ is visible from $y_{C'}$. Then $F$ is visible from $y_{C'}$.
  
  Conversely, suppose that $F$ is visible from $y_{C'}$.
  Denote by $\mathcal{B} \subseteq \supp(F)$ the set of hyperplanes $|G|$ where $G\supseteq F$ is a facet of $E'$.
  All the facets $G \supseteq F$ of $E'$ are visible from $y_{C'}$, so the hyperplanes $|G|$ separate $C'$ and $E'$.
  In other words, $\mathcal{B} \subseteq s(C',E')$.
  In the central arrangement $\A^{X_C}_{|F|} = \supp(F)$, the chambers $\pi_{|F|}(C')$ and $\pi_{|F|}(E')$ are therefore opposite to each other, and $\mathcal{B}$ is the set of their walls.
  Then every hyperplane in $\supp(F)$ separates $C'$ and $E'$.
\end{proof}

Fix an arbitrary point $y_{C'}$ in the interior of $C'$.
By the previous lemma, the faces $F$ given by Corollary \ref{corollary:visible} are exactly the faces of $E'$ that are visible from $y_{C'}$. See Figure \ref{fig:visible} for an example.

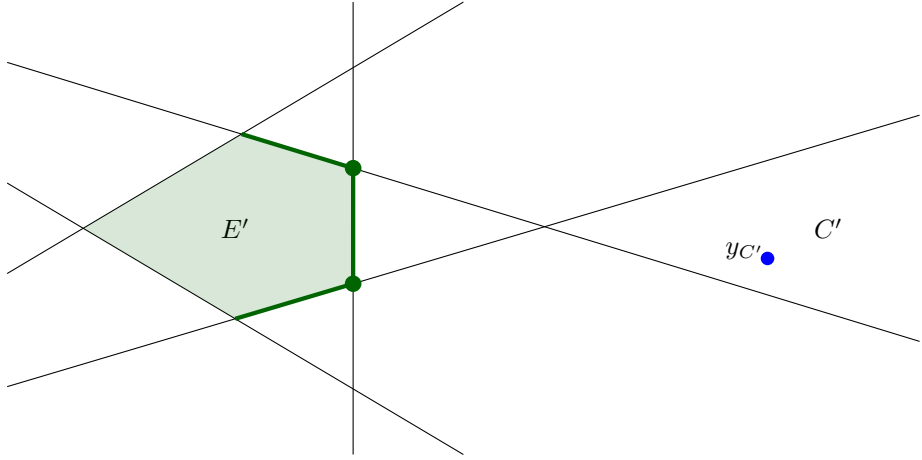
\begin{figure}
   \begin{tikzpicture}
    \draw[name path=L1] (0,0.9) -- (12,4.5);
    \draw[name path=L3] (0,5.2) -- (12,1.5);
    \draw[name path=L2] (4.55,0) -- (4.55,6);
    \draw[name path=L4] (0,2.4) -- (6,6);
    \draw[name path=L5] (0,3.6) -- (6,0);
    
    \path[name intersections={of=L1 and L2, by=E0}];
    \path[name intersections={of=L3 and L2, by=E1}];
    \path[name intersections={of=L1 and L5, by=E2}];
    \path[name intersections={of=L3 and L4, by=E3}];
    \path[name intersections={of=L4 and L5, by=E4}];

    \node (E) at (3,3) {$E'$};
    \node (C) at (10.8,3) {$C'$};
    
    \node at (E2) {};
    \node at (E3) {};
    \node at (E4) {};
    
    \begin{scope}[every node/.style={circle,inner sep=2.2pt,fill=darkgreen}]
        \node at (E0) {};
        \node at (E1) {};
    \end{scope}
    
    \node at (9.7,2.7) {$y_{C'}$};
    
    \begin{scope}[every node/.style={circle,inner sep=1.8pt,fill=blue}]
      \node at (10,2.6) {};
    \end{scope}

    \fill[opacity=0.15, color=darkgreen] (E0) -- (E1) -- (E3) -- (E4) -- (E2);
    
    \begin{scope}[every path/.style={draw=darkgreen,ultra thick, fill=darkgreen}]
     \draw (E0) -- (E1);
     \draw (E0) -- (E2);
     \draw (E1) -- (E3);
    \end{scope}

  \end{tikzpicture}
  
  \caption{The faces of $E'$ that are visible from a point $y_{C'}$ in the interior of $C'$.}
  \label{fig:visible}
\end{figure}

The idea now is that, if $E'$ is bounded, the boundary of $E'$ is shellable and we can use a shelling to construct an acyclic matching on the set of visible faces.
We first need to reduce to the case of a bounded chamber (i.e.\ a polytope).

\begin{lemma}\label{lemma:unb-bounded}
  There exists a finite set $\A'$ of hyperplanes in $X_C$, and a bounded chamber $\tilde E \subseteq E'$ of the hyperplane arrangement $\A' \cup \A^{X_C}$, such that the poset of faces of $\tilde{E}$ that are visible from $y_{C'}$ is isomorphic to the poset of faces of $E'$ that are visible from $y_{C'}$.
\end{lemma}

\begin{proof}
  Let $X_C \cong \R^k$.
  Let $Q$ be a finite set of points which contains $y_{C'}$ and a point in the relative interior of each visible face of $E'$.
  For $i=1,\dots, k$, define $q_i \in \R$ as the minimum of all the $i$-th coordinates of the points in $Q$, and $q^i$ as the maximum. 
 
  Choose $\A'$ as the set of the $2k$ hyperplanes of the form $\{ x_i=q_i-1 \}$ and $\{ x_i=q^i+1 \}$, for $i=1,\dots, k$.
  Let $\tilde{E}$ be the chamber of $\A^{X_C} \cup \A'$ that contains $Q \setminus \{y_{C'}\}$.
  By construction, $\tilde{E}$ is bounded and is contained in $E'$.
  See Figure \ref{fig:Lemma} for an example.
  
  The walls of $E'$ and of $\tilde E$ are related as follows: $\W_{\tilde E} = \W_{E'} \cup \A''$ for some $\A'' \subseteq \A'$.
  The hyperplanes in $\W_{E'}$ separate $y_{C'}$ and $\tilde E$, whereas the hyperplanes in $\A''$ do not.
  This means that a facet $\tilde G$ of $\tilde E$ is visible if and only if $|\tilde G| \in \W_{E'}$.
  
  There is a natural order-preserving (and rank-preserving) injection $\varphi$ from the set $\mathcal{V}$ of the visible faces $F$ of $E'$ to the set of faces of $\tilde{E}$, which maps a face $F$ to the unique face $\tilde F$ of $\tilde E$ such that $F \cap Q \subseteq \tilde F \subseteq F$.
  We want to show that the image of $\varphi$ coincides with the set of visible faces of $\tilde E$.
  
  Consider a facet $\tilde G$ of $\tilde E$.
  Then $\tilde G$ is in the image of $\varphi$ if and only if $|\tilde G| \not\in \A''$, which happens if and only if $\tilde G$ is visible.

  Consider now a generic face $\tilde F$ of $\tilde E$.
  If $\tilde F = \varphi(F)$ for some $F \in \mathcal{V}$, then $Q \cap F \subseteq \tilde F$ and so $\tilde F$ is not contained in any hyperplane of $\A''$.
  Then all the facets $\tilde G \supseteq \tilde F$ of $\tilde E$ are visible, and so $\tilde F$ is visible.
  Conversely, if $\tilde F$ is not in the image of $\varphi$, then $\tilde F$ is contained in some hyperplane of $\A''$ and therefore also in some non-visible facet $\tilde G$. Then $\tilde F$ is not visible.
\end{proof}

\begin{figure}
   \begin{tikzpicture}
    \draw[name path=L1] (0,0.9) -- (12,4.5);
    \draw[name path=L3] (0,5.2) -- (12,1.5);
    \draw[name path=L2] (4.55,0) -- (5.55,6);
    
    \path[name intersections={of=L1 and L2, by=E0}];
    \path[name intersections={of=L3 and L2, by=E1}];

    \node (E') at (3,3) {$E'$};
    \node (C') at (11.4,3) {$C'$};
    \node at (9.7,2.7) {$y_{C'}$};

    \begin{scope}[every path/.style={dashed, thick}]
      \draw (0.7,0) -- (0.7,6);
      \draw (10.7,0) -- (10.7,6);
      \draw (0,0.64) -- (12,0.64);
      \draw (0,5) -- (12,5);
    \end{scope}

    \begin{scope}[every path/.style={draw=darkgreen, ultra thick, fill=darkgreen}]
      \draw (E0) -- (E1);
      \draw (E0) -- (0,0.9);
      \draw (E1) -- (0,5.2);
    \end{scope}
    
    \fill[opacity=0.15, color=darkgreen] (E0) -- (E1) -- (0,5.2)-- (0,0.9);
    
    \begin{scope}[every node/.style={circle,inner sep=1.8pt,fill=blue}]
        \node at (E0) {};
        \node at (E1) {};
        \node at (5.05,3) {};
        \node at (1.5,1.34) {};
        \node at (3,4.27) {};
        \node at (2,3.5) {};
        \node at (10,2.6) {};
    \end{scope}
  \end{tikzpicture}
  
  \caption{Construction of the bounded chamber $\tilde E \subseteq E'$ in the proof of Lemma \ref{lemma:unb-bounded}.
  The points of $Q$ are highlighted, and the hyperplanes of $\A'$ are dashed.}
  \label{fig:Lemma}
\end{figure}
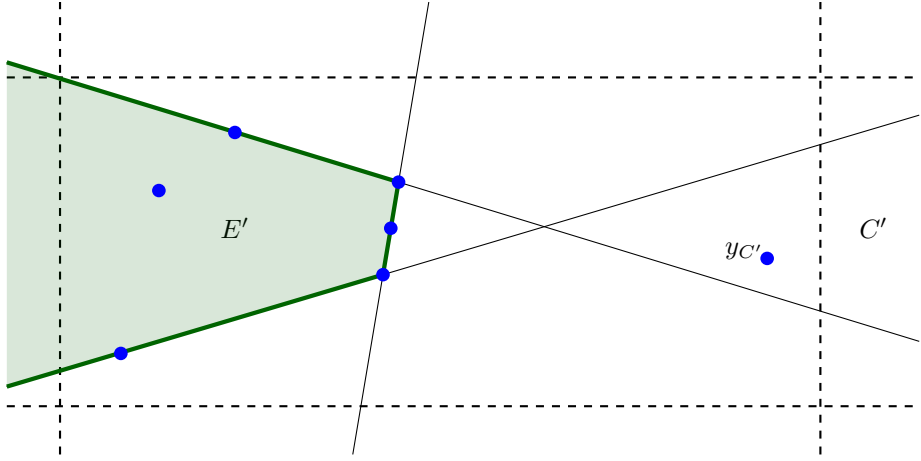

We now show that the poset of visible faces of a polytope admits an acyclic matching such that no face is critical.
We will use this result on the polytope $\tilde E$, in order to obtain a matching on the fiber $\eta_C^{-1}(E)$.

\begin{theorem}\label{teo:fiber}
  Let $X$ be a $k$-dimensional polytope in $\R^k$, and let $y \in \R^k$ be a point outside $X$ that does not lie in the affine hull of any facet of $X$.
  Then there exists an acyclic matching on the poset of faces of $X$ visible from $y$ such that no face is critical.
\end{theorem}

\begin{proof}
  By \cite[Theorem 8.12]{ziegler2012lectures} and \cite[Lemma 8.10]{ziegler2012lectures}, there is a shelling $G_1,\dots,G_s$ of $\partial X$ such that the facets visible from $y$ are the last ones.
  Suppose that $G_t,G_{t+1},\dots, G_s$ are the visible facets.
Notice that there is at least one visible facet and at least one non-visible facet.
  In particular, the first facet $G_1$ is not visible and the last facet $G_s$ is visible.
  In other words, we have $2\leq t \leq s$.
  
  In \cite[Proposition 1]{delucchi2008shelling} it is proved that a shelling of a regular CW complex $Y$ induces an acyclic matching on the poset of cells $(P, <)$ of $Y$ (augmented with the empty face $\emptyset$), with critical cells corresponding to the spanning facets of the shelling.
  In our case, $Y = \partial X$ is a regular CW decomposition of a sphere, so the only spanning facet of a shelling is the last one (see for example \cite[Lemma 2.13]{delucchi2008shelling}).
  
  Let $\M$ be an acyclic matching on $\partial X$ induced by the shelling $G_1,\dots,G_s$, as in \cite[Proposition 1]{delucchi2008shelling}.
  We claim that the construction of \cite{delucchi2008shelling} produces a matching which is homogeneous with respect to the grading $\varphi\colon (P, <) \to \{1,\dots,s\}$ given by
  \[ \varphi(F) = \min \{ i \in \{1,\dots,s\} \mid F \leq G_i \}. \]
  To prove this, we need to briefly go through the construction of $\M$.
  The first step \cite[Lemma 2.10]{delucchi2008shelling} is to construct a total order $\sqsubset_i$ on each $P_i$ (the set of faces of codimension $i$).
  The order $\sqsubset_0$ is simply the shelling order of the facets.
  It follows from the recursive construction of $\sqsubset_i$ that each $\varphi|_{P_i} \colon (P_i, \sqsubset_i) \to \{1,\dots,s\}$ is order-preserving.
  Then the linear extension $\vartriangleleft$ of $P$ constructed in \cite[Definition 2.11]{delucchi2008shelling} is such that $\varphi\colon (P,\vartriangleleft) \to \{1,\dots,s\}$ is also order-preserving.
  By construction of the matching \cite[Lemma 2.12]{delucchi2008shelling}, if $(p,q) \in \M$ (with $p \geq q$) then $p \vartriangleleft q$.
  From this we obtain $\varphi(p) \geq \varphi(q)$ and $\varphi(p) \leq \varphi(q)$, so $\varphi(p) = \varphi(q)$.
  Therefore the matching is homogeneous with respect to $\varphi$.
  
  The set of visible faces of $X$ is $\varphi^{-1}(\{t,\dots,s\}) \cup \{ X \}$.
  Notice that the empty face $\emptyset$ belongs to $\varphi^{-1}(1)$, so it does not appear in $\varphi^{-1}(\{t,\dots,s\})$ because $t \geq 2$.
  
  Let $\M'$ be the restriction of $\M$ to $\varphi^{-1}(\{t,\dots,s\})$.
  This is an acyclic matching on $\varphi^{-1}(\{t,\dots,s\})$ with exactly one critical face, the facet $G_s$.
  Then $\M' \cup \{(X, G_s) \}$ is an acyclic matching on the poset of visible faces of $X$ such that no face is critical.
\end{proof}

We are finally able to attach the matchings on the fibers $\eta_C^{-1}(E)$, using the previous results of this section.

\begin{theorem}
  \label{thm:matching}
  Let $\A$ be a locally finite hyperplane arrangement, and let $\dashv$ be a valid order of the set of chambers $\C(\A)$. For every chamber $C \in \C(\A)$, there exists a proper acyclic matching on $N(C)$ such that the only critical cell is $\cell{C, F_C}$. The union of these matchings forms a proper acyclic matching on $\sal(\A)$ with critical cells in bijection with the chambers.
\end{theorem}

\begin{proof}
  Consider the map $\eta\colon \sal(\A) \to \C \times \C$ defined as
  \[
    \eta(\cell{D,F}) = (C, D^F),
  \]
  where $C \in \C$ is the chamber such that $\cell{D,F} \in N(C)$.

  Corollary \ref{corollary:visible} provides a description of the non-empty fibers $\eta^{-1}(C,E)$, since by definition we have $\eta^{-1}(C,E)=\eta_C^{-1}(E)$.
  By Lemma \ref{lemma:special-fiber}, we know that for every $C \in \C$ there is exactly one non-empty fiber such that $E \cap X_C = C \cap X_C$, and this fiber contains the single cell $\cell{C,F_C}$.
  By Lemma \ref{lemma:visible-faces} and Lemma \ref{lemma:unb-bounded}, every other non-empty fiber $\eta^{-1}(C,E)$ is isomorphic to the poset of visible faces of some polytope in $X_C$ (with respect to some external point not lying on the affine hull of the facets).
  Finally, by Theorem \ref{teo:fiber}, this poset admits an acyclic matching with no critical faces.
  
  We want to use the Patchwork Theorem (Theorem \ref{teo:patchwork}) to attach these matchings together.
  To do so, we first need to define a partial order on $\C \times \C$ that makes $\eta$ a poset map.
  The order $\leq$ on $\C \times \C$ is the transitive closure of:
  \[
    (C',E') \leq (C,E) \quad \text{if and only if} \quad C' \dashveq C \text{ and } E' \leq_C E
  \]
  (here we denote by $\dashveq$ the ``less than or equal to'' with respect to the total order $\dashv$).
  
  To prove that $\eta$ is a poset map, suppose to have $\cell{D',F'} \leq \cell{D,F}$ in $\sal(\A)$.
  Let $\eta(\cell{D',F'})=(C',E')$ and $\eta(\cell{D,F})=(C,E)$. 
  Since $S(C)$ is a lower ideal of $\sal(\A)$, we immediately obtain that $\cell{D',F'} \in S(C)$ and thus $C' \dashveq C$. Then, Lemma \ref{lemma:eta-order-preserving} implies that $E' \leq_C E$.
  Therefore $(C',E') \leq (C,E)$.
  
  By the Patchwork Theorem, the union of the matchings on the fibers of $\eta$ forms an acyclic matching on $\sal(\A)$, with critical cells in bijection with the chambers.
  
  We now need to prove that this matching is proper.
  To do so, we prove that the $(\C \times \C)$-grading $\eta$ is compact.
  Since every fiber $\eta^{-1}(C,E)$ is finite by Lemma \ref{lemma:fiber}, we only need to show that the poset $(\C \times \C)_{\leq (C,E)}$ is finite for every pair of chambers $(C,E)$.
  
  We prove this by double induction, first on the chamber $C$ (with respect to the order $\dashv$) and then on $m=|s(C,E)|$.
  The base case $C=C_0$ and $m=0$ is trivial, since $E=C_0$. 
  
  We want now to prove the induction step.
  Given a pair $(C,m) \in \C \times \N$, suppose that the claim is true for every pair $(C',m')$ such that either $C' \dashv C$, or $C'=C$ and $m'<m$.
  For every chamber $E$ with $|s(C,E)| = m$ we have that 
  \begin{equation*}
    (\C \times \C)_{\leq (C,E)} \; = \!\!\! \bigcup_{
      \substack{C' \dashveq C \\ E' \leq_C E \\ (C',E')\neq (C,E)}
    }{\!\!\!\!\!\!\! (\C \times \C)_{\leq (C',E')}} \;\;\cup\; \{(C,E)\}.
  \end{equation*}
  This is a union of a finite number of sets, and by the induction hypothesis every set $(\C \times \C)_{\leq (C',E')}$ is finite.
  Therefore the set $(\C \times \C)_{\leq (C,E)}$ is finite.
    
  By the Patchwork Theorem, the matchings on the fibers $\eta^{-1}(C,E)$ can be attached together to form a proper acyclic matching on $\sal(\A)$.
  By construction, this matching is a union of proper acyclic matchings on the subposets $N(C)$ for $C \in \C$, each of them having $\cell{C,F_C}$ as the only critical cell.
\end{proof}

We end this section with a few remarks.
We are not going to use them in the rest of this paper, but they are interesting by themselves (especially in relation with \cite{delucchi2008shelling}).

The first remark is that, without the need of a valid order, the results of this section allow us to obtain a proper acyclic matching on $S(C_0)$ (for any chamber $C_0 \in \C$) with the single critical cell $\cell{C_0,C_0}$.
This is because $N(C_0) = S(C_0)$, and in the construction of the matching on $N(C_0)$ we do not use the existence of a valid order that begins with $C_0$.
As noted in Section \ref{sec:decomposition}, there is a natural poset isomorphism $S(C_0) \cong \F$ for every chamber $C_0 \in \C$.
Then the existence of an acyclic matching on $S(C_0)$ can be stated purely in terms of $\F$, without speaking of the Salvetti complex.
This result appears in \cite[Theorem 3.6]{delucchi2008shelling} in the case of the face poset of an oriented matroid.

\begin{provedtheorem}
  \label{thm:matching-faces}
  Let $\A$ be a locally finite hyperplane arrangement.
  For every chamber $C \in \C(\A)$, there is a proper acyclic matching on the poset of faces $\F(\A)$ such that $C$ is the only critical face.
\end{provedtheorem}

The second remark is that, given a valid order $\dashv$ of $\C$ and a chamber $C \in \C$, the poset $N(C)$ is isomorphic to $\F(\A^{X_C})$.
This is the analogue of \cite[Lemma 4.20]{delucchi2008shelling}.

\begin{lemma}
  \label{lemma:NC-F}
  Suppose that $\dashv$ is a valid order of $\C$. For every chamber $C \in \C$ there is a poset isomorphism
  \[ N(C) \cong \F\big(\A^{X_C}\big). \]
\end{lemma}

\begin{proof}
  The isomorphism from $N(C)$ to $\F\big(\A^{X_C}\big)$ sends a cell $\cell{D,F} \in N(C)$ to the face $F$, which is in $\F(\A^{X_C})$ by Lemma \ref{lemma:NC}.
  The inverse map sends a face $F \in \F(\A^{X_C})$ to the cell $\cell{C.F,F}$, which is in $N(C)$ by definition of $S(C)$ and by Lemma \ref{lemma:NC}.
  These maps are order-preserving.
\end{proof}

Together, Lemma \ref{lemma:NC-F} and Theorem \ref{thm:matching-faces} give an alternative (but equivalent) construction of our matching on $\sal(\A)$, closer to the approach of \cite{delucchi2008shelling}.

 \section{Euclidean matchings}
\label{sec:euclidean}

In this section we are going to construct a valid order $\eu$ of the set of chambers $\C$, for any locally finite arrangement $\A$, using the Euclidean distance in $\R^n$.
Then we are going to prove that the matching induced by this order (given by Theorem \ref{thm:matching}) yields a minimal Morse complex.

Denote by $d$ the Euclidean distance in $\R^n$.
Also, if $K$ is a closed convex subset of $\R^n$, denote by $\rho_K(x)$ the projection of a point $x\in \R^n$ onto $K$.
The point $\rho_K(x)$ is the unique point $y \in K$ such that $d(x,y) = d(x, K)$.

The first step is to prove that there exist a lot of \emph{generic points} with respect to the arrangement $\A$.
For this, we need the following technical lemma.
By \emph{measure} we always mean the Lebesgue measure in $\R^n$.

\renewcommand{\S}{\mathcal{S}}
\begin{lemma}\label{lemma:measure}
  Let $K_1$ and $K_2$ be two closed convex subsets of $\R^n$.
  Let
  \[ \S = \{ x \in \R^n \mid d(x,K_1) = d(x,K_2) \text{ and } \rho_{K_1}(x) \neq \rho_{K_2}(x) \}. \]
  Then $\S$ has measure zero.
\end{lemma}

\begin{proof}
  This proof was suggested by Federico Glaudo.
  Let $d_i(x) = d(x,K_i)$ for $i=1,2$.
  Each function $d_i \colon \R^n \to \R$ is differentiable on $\R^n \setminus K_i$ by \cite[Lemma 2.19]{giaquinta2012convex}, and its gradient in a point $x \not\in K_i$ is the versor with direction $x-\rho_{K_i}(x)$.
  
  Let $f(x) = d_1(x) - d_2(x)$.
  Denote by $A$ the open set of points $x \in \R^n \setminus (K_1 \cup K_2)$ such that $\rho_{K_1}(x) \neq \rho_{K_2}(x)$.
  On this set, the function $f$ is differentiable and its gradient does not vanish.
  It is known that the gradient of $f$ must vanish almost everywhere on $A\cap f^{-1}(0)$ \cite[Corollary 1 of Section 3.1]{lawrence1992measure}, hence $A\cap f^{-1}(0)$ has measure zero.
  
  It is easy to check that the points in $K_1 \cup K_2$ cannot belong to $\S$.
  Then $\S = A \cap f^{-1}(0)$ has measure zero.
\end{proof}

\newcommand{\G}{\mathcal{G}}
\newcommand{\T}{\mathcal{T}}
\begin{lemma}[Generic points]\label{lemma:eucpt}
  Given a locally finite hyperplane arrangement $\A$ in $\R^n$, let $\G \subseteq \R^n$ be the set of points $x \in \R^n$ such that:
  \begin{enumerate}[(i)]
    \item for every $C,C' \in \mathcal{C}$ with $d(x,C)=d(x,C')$, we have $\rho_C(x) = \rho_{C'}(x) \in C \cap C'$;
    
    \item for every $L,L' \in \L$ with $L' \subsetneq L$, we have $d(x,L') > d(x,L)$.
\end{enumerate}
  Then the complement of $\G$ has measure zero.
  In particular, $\G$ is dense in $\R^n$.
\end{lemma}

\begin{proof}
  Given $C,C' \in \C$, let $\S_{C,C'}$ be the set of points $x \in \R^n$ such that $d(x,C_1) = d(x,C_2)$ and $\rho_{C_1}(x) \neq \rho_{C_2}(x)$.
  By Lemma \ref{lemma:measure}, every $\S_{C,C'}$ has measure zero.
  
  Similarly, for every $L,L' \in \L$ with $L' \subsetneq L$, denote by $\T_{L,L'}$ the set of points $x \in \R^n$ such that $d(x, L') = d(x,L)$.
  We have that $\T_{L,L'}$ is an affine subspace of $\R^n$ of codimension at least $1$, and in particular it has measure zero.
  
  The complement of $\G$ is the union of all the sets $\S_{C,C'}$ for $C,C' \in \C$ and $\T_{L,L'}$ for $L,L' \in \L$ with $L' \subsetneq L$.
  This is a finite or countable union of sets of measure zero, hence it has measure zero.
\end{proof}

We call \emph{generic points} the elements of $\G$, as defined in Lemma \ref{lemma:eucpt}.
Notice that, by condition (ii) with $L=\R^n$, a generic point must lie in the complement of $\A$.

\begin{remark}
 An alternative proof of the previous Lemma can be found in \cite[Part III, Section 3.1 and Part I, Section 2.2]{goresky1988stratified}, within the more general setting of density of Morse functions.
\end{remark}

\begin{remark}
  \label{rmk:generic-point-equivalent}
  An equivalent definition of a generic point is the following: $x_0 \in \R^n$ is generic with respect to $\A$ if and only if every flat of $\A$ has a different distance from $x_0$.
  Indeed, this definition immediately implies condition (ii) of Lemma \ref{lemma:eucpt}.
  It also implies condition (i), because for any chamber $C$ we have $d(x_0, C) = d(x_0, L)$ where $L$ is the smallest flat that contains $\rho_C(x_0)$.
  Conversely, suppose that $x_0$ satisfies both conditions (i) and (ii).
  Given two flats $L, L' \in L$ with $d(x_0, L) = d(x_0, L')$, by condition (ii) the projections $\rho_L(x_0)$ and $\rho_{L'}(x_0)$ must lie in the relative interior of faces $F, F' \in \F$ with $|F| = L$ and $|F'| = L'$.
  Defining $C$ as the chamber containing $F$ and with the greatest distance from $x_0$, we immediately obtain that $\rho_L(x_0)=\rho_C(x_0)$.
  If $C'$ in defined in the same way (using $F'$ and $L'$), the chambers $C$ and $C'$ violate condition (ii) unless $L=L'$.
  With this equivalent definition, it is possible to prove Lemma \ref{lemma:eucpt} in an alternative way without using Lemma \ref{lemma:measure} (cf.\ Lemma \ref{lemma:add-generic-hyperplane}).
\end{remark}

We are now able to define Euclidean orders.

\begin{definition}[Euclidean orders]
  A total order $\eu$ of the set of chambers $\C$ is \emph{Euclidean} if there exists a generic point $x_0$ such that $C \eu C'$ implies that $d(x_0,C)\leq d(x_0,C')$.
  The point $x_0$ is called a \emph{base point} of the Euclidean order~$\eu$.
  
  Notice that a Euclidean order is any linear extension of the partial order on $\C$ given by $C < C'$ if $d(x_0, C) < d(x_0, C')$, for some fixed generic point $x_0 \in \R^n$.
\end{definition}

In particular, for every generic point $x_0$ there exists at least one Euclidean order with $x_0$ as a base point.
Since the set of generic points is dense, we immediately get the following corollary.

\begin{corollary}
  For every chamber $C_0 \in \C$, there exists a Euclidean order $\eu$ that starts with $C_0$.
\end{corollary}

\begin{proof}
  It is enough to take the base point $x_0$ in the interior of the chamber $C_0$.
\end{proof}

See Figure \ref{fig:eucl-order} for an example of a Euclidean order.
We now prove that every Euclidean order is valid, in the sense of Definition \ref{def:valid-order}.

\begin{theorem}
  \label{thm:euclidean-valid}
  Let $\eu$ be a Euclidean order with base point $x_0$.
  For every chamber $C$, let $x_C = \rho_C(x_0)$ and let $F_C$ be the smallest face of $C$ that contains $x_C$.
  Then $\J(C)$ is the principal upper ideal generated by $X_C = |F_C|$.
  Therefore $\eu$ is a valid order.
\end{theorem}

\begin{proof}
  First we want to prove that $X_C \in \J(C)$.
  This is equivalent to proving that for every chamber $C' \eu C$ there exists a hyperplane $H \in \supp(X_C)\cap s(C,C')$.
  We have that $\rho_{X_C}(x_0) = x_C$ because $F_C$ is the smallest face that contains $x_C$.
  Thus it is also true that $\rho_{\pi_{X_C}(C)}(x_0)=x_C$.
  Given a chamber $C'\eu C$, we have two possibilities.
  \begin{itemize}
    \item $d(x_0,C')<d(x_0,C)$.
    Then $C' \nsubseteq \pi_{X_C}(C)$, because all the points of $\pi_{X_C}(C)$ have distance at least $d(x_0,C)$ from $x_0$.
    This means that there exists a hyperplane $H \in \supp(X_C)=\A_{X_C}$ which separates $C$ and $C'$.
  
    \item $d(x_0,C')=d(x_0,C)$.
    Since $x_0$ is a generic point, we have that $x_C = x_{C'} \in C \cap C'$.
    Then $F_C$ is a common face of $C$ and $C'$, and every hyperplane in $s(C,C')$ contains $F_C$.
 \end{itemize}

  Now we want to prove that $X \subseteq X_C$ for every $X \in \J(C)$.
  Suppose by contradiction that $X \nsubseteq X_C$ for some $X \in \J(C)$.
  In particular, $X_C \neq \R^n$ and thus $x_0 \neq x_C$.
  We first prove that $\supp(X_C \cup X)$ is non-empty.
  
  Let $C'$ be the chamber of $\A$ such that $x_0 \in \pi_{X_C}(C')$ and $C' \prec F_C$.
  Since $x_C \in X_C \subseteq \pi_{X_C}(C')$, the entire line segment $\ell$ from $x_0$ to $x_C$ is contained in $\pi_{X_C}(C')$.
  Thus there is a neighborhood of $x_C$ in $\ell$ which is contained in $C'$, hence $d(x_0, C') < d(x_0, x_C)$ and therefore $C' \eu C$.
  Since $X \in \J(C)$, there exists a hyperplane $H \in \supp(X) \cap s(C,C')$.
  We also have that $F_C \subseteq C \cap C'$, and thus $X_C \subseteq H$.
 
  Consider now the flat $X' = \bigcap\{Z \in \L \mid X_C \cup X \subseteq Z\}$, i.e.\ the meet of $X_C$ and $X$ in $\L$.
  The flat $X'$ is contained in the hyperplane $H$ constructed above, so in particular $X' \neq \R^n$.
  In addition, since $X \nsubseteq X_C$, $X'$ is different from $X_C$.
  Then the point $y_0=\rho_{X'}(x_0)$ is different from $x_C$, and we have $d(x_0, y_0) < d(x_0, x_C)$, because $x_0$ is generic (see condition (ii) of Lemma \ref{lemma:eucpt}).
  Let $F$ be the smallest face that contains the line segment $[x_C, \, x_C+ \epsilon (y_0-x_C)]$ for some $\epsilon>0$.
  By construction, for every chamber $C''$ such that $C'' \preceq F$ we have that $C'' \eu C$.
  This holds in particular for $C'' = C.F$.
Then we have $\supp(F) \cap s(C,C'') = \emptyset$.
  
  Since $X \in \J(C)$ and $C'' \eu C$, there exists a hyperplane $H \in \supp(X) \cap s(C,C'')$.
  By construction, $x_C \in C \cap C''$ and then $X_C$ is contained in every hyperplane of $s(C,C'')$.
  In particular, $X_C \subseteq H$.
  Therefore $X_C \cup X \subseteq H$, which means that $H \in \supp(X_C \cup X) \subseteq \supp(X')$.
  Both $x_C$ and $y_0$ belong to $X'$, hence $F \subseteq X'$.
  Putting everything together, we get $H \in \supp(X') \cap s(C,C'') \subseteq \supp(F) \cap s(C,C'') = \emptyset$.
  This is a contradiction.
\end{proof}

Since Euclidean orders are valid, we are able to construct acyclic matchings on the Salvetti complex of any arrangement.

\begin{definition}[Euclidean matchings]
  \label{def:euclidean-matching}
  Let $\A$ be a locally finite hyperplane arrangement in $\R^n$.
  We say that an acyclic matching $\M$ on $\sal(\A)$ is a \emph{Euclidean matching} with base point $x_0 \in \R^n$ if:
  \begin{enumerate}[(i)]
    \item the point $x_0$ is generic with respect to $\A$;
    \item $\M$ is homogeneous with respect to the poset map $\eta \colon \sal(\A) \to \C \times \C$ induced by a Euclidean order $\eu$ with base point $x_0$ (defined as in the proof of Theorem \ref{thm:matching});
    \item there is exactly one critical cell $\cell{C, F_C}$ for every chamber $C \in \C$, where $F_C$ is the smallest face of $C$ that contains $\rho_C(x_0)$.
  \end{enumerate}
  Notice that, by condition (ii), a Euclidean matching is also proper.
\end{definition}

\begin{theorem}
  \label{thm:euclidean-matching}
  Let $\A$ be a locally finite arrangement in $\R^n$.
  For every generic point $x_0 \in \R^n$, there exists a Euclidean matching on $\sal(\A)$ with base point $x_0$.
\end{theorem}

\begin{proof}
  It follows from Theorems \ref{thm:matching} and \ref{thm:euclidean-valid}.
\end{proof}

\begin{remark}
  \label{rmk:matching-independent}
  For a given generic point $x_0$, there might be more than one Euclidean order $\eu$ with base point $x_0$.
  Nonetheless, all Euclidean orders with a given base point produce the same faces $F_C$ (by Theorem \ref{thm:euclidean-valid}) and the same critical cells (by Theorem \ref{thm:matching}).
  The decomposition
  \[ \sal(\A) = \bigsqcup_{C \in \C} N(C) \]
  also depends only on $x_0$ (by Lemma \ref{lemma:NC}), and therefore the definition of a Euclidean matching is not influenced by the choice of $\eu$ (once the base point $x_0$ is fixed).
\end{remark}

We are going to prove that a Euclidean matching yields a minimal Morse complex.
In order to do so, we first prove two lemmas about generic points.

\begin{lemma}
  \label{lemma:generic-point-subarrangement}
  Let $x_0 \in \R^n$.
  If $x_0$ is generic with respect to an arrangement $\A$, then it is also generic with respect to any subarrangement $\A' \subseteq \A$.
\end{lemma}

\begin{proof}
  Condition (i) for $\A'$ holds because a chamber of $\A'$ is a union of chambers of $\A$.
  Condition (ii) follows from the fact that $\L(\A') \subseteq \L(\A)$.
\end{proof}

\begin{lemma}
  \label{lemma:add-generic-hyperplane}
  Let $\A$ be a locally finite arrangement in $\R^n$, and let $x_0 \in \R^n$ be a generic point with respect to $\A$. Let $\H \subseteq (\R^n\setminus \{0\}) \times \R \subseteq \R^{n+1}$ be the set of elements $(a_1,\dots,a_n,b)$ such that $x_0$ is generic also with respect to the arrangement $\A \cup \{H\}$, where $H$ is the hyperplane defined by the equation $a_1x_1 + \dots + a_nx_n = b$.
  Then the complement of $\H$ in $\R^{n+1}$ has measure zero.
  In particular, $\H$ is dense in $\R^{n+1}$.
\end{lemma}

\begin{proof}
  In this proof we use the equivalent definition of a generic point given in Remark \ref{rmk:generic-point-equivalent}.
  Assume that $H$ intersects generically every flat $X \in \L(\A)$, i.e.\ $\codim(X \cap H) = \codim(X) + 1$.
  This condition excludes a subset of measure zero in $\R^{n+1}$.
  
  Since $x_0$ is generic with respect to $\A$, the distances between $x_0$ and the flats of $\A$ are all distinct.
  Consider now a flat of $\A \cup \{H\}$ of the form $X \cap H$, for some flat $X \in \L(\A)$ of dimension $\geq 1$.
  The squared distance $d^2(x_0, X \cap H)$ is a rational function of the coefficients $(a_1,\dots,a_n,b)$ that define $H$.
  
  Given two flats $X,Y \in \L(\A)$ with $\dim(X) \geq 1$, the condition $d^2(x_0, X \cap H) = d^2(x_0, Y)$ can be written as a polynomial equation $p(a_1,\dots,a_n,b)=0$.
  This equation is not satisfied if $d(x_0, H) > d(x_0, Y)$, therefore the polynomial $p$ is not identically zero.
  Then the zero locus of $p$ has measure zero.
  
  Similarly, given two flats $X,Y \in \L(\A)$ with $\dim(X) \geq 1$ and $\dim(Y) \geq 1$, the condition $d^2(x_0, X \cap H) = d^2(x_0, Y \cap H)$ can be written as a polynomial equation $q(a_1,\dots,a_n,b) = 0$.
  Up to exchanging $X$ and $Y$, we can assume that $\rho_X(x_0) \not \in Y$, because $d(x_0,X) \neq d(x_0, Y)$.
  If $H$ is the hyperplane orthogonal to the vector $\rho_X(x_0) - x_0$ that passes through $\rho_X(x_0)$, then we have $d(x_0, X\cap H) = d(x_0, X)$ and $d(x_0, Y \cap H) > d(x_0, H) = d(x_0, X)$ (the inequality is strict because $Y$ does not contain $\rho_H(x_0) = \rho_X(x_0)$).
  Therefore the polynomial $q$ is not identically zero, and the zero locus of $q$ has measure zero.
  
  Thus the complement of $\H$ is contained in a finite or countable union of sets of measure zero, and hence it has measure zero.
\end{proof}

\begin{theorem}[Minimality]
  \label{thm:minimality}
  Let $\A$ be a locally finite hyperplane arrangement in $\R^n$, and let $\M$ be a Euclidean matching on $\sal(\A)$ with base point $x_0$.
  Then the associated Morse complex $\sal(\A)_{\M}$ is minimal (i.e., all the incidence numbers va\-nish).
\end{theorem}

\begin{proof}
  If the arrangement $\A$ is finite, it is well known that the sum of the Betti numbers of $\sal(\A)$ is equal to the number of chambers \cite{orlik1980combinatorics, zaslavsky1997facing}.
  By Theorem \ref{thm:matching}, the critical cells of $\M$ are in bijection with the chambers.
  Thus the Morse complex is minimal.

  Suppose from now on that $\A$ is infinite.
  Fix a chamber $C \in \C$, and consider the associated critical cell $\cell{C,F_C} \in N(C)$.
  Recall from the proof of Theorem \ref{thm:matching} the definition of the poset map $\eta \colon \sal(\A) \to \C \times \C$, and let $(C,E)=\eta(\cell{C,F_C})$.
  Since the matching is proper, the set $\eta^{-1}((\C \times \C)_{\leq (C,E)})$ is finite.
  
  Consider now the finite set of faces
  \[
    \U = \; \{ F \in \F \mid \cell{D,F} \in \eta^{-1}((\C \times \C)_{\leq (C,E)}) \text{ for some chamber } D \in \C \}.
  \]
  Let $B \subseteq \R^n$ be an open Euclidean ball centered in $x_0$ that contains the projection $\rho_F(x_0)$ for every face $F \in \U$.
  Let $\bar \A$ be a set of $n+1$ hyperplanes that do not intersect $B$, such that: $x_0$ is still generic with respect to $\A \cup \bar \A$; the chamber $K$ of the arrangement $\bar \A$ containing $B$ is bounded. Such an arrangement $\bar \A$ exists thanks to Lemma \ref{lemma:add-generic-hyperplane}.
  Consider the finite arrangement
  \[ \A' = \{ H \in \A \mid H \cap K \neq \emptyset \} \cup \bar \A, \]
  and let $\F_K \subseteq \F(\A)$ be the set of faces of $\A$ that intersect the interior of $K$.
  Notice that, by construction, we have $\U \subseteq \F_K$.
  In addition, there is a natural order-preserving and rank-preserving injection $\varphi\colon \F_K \to \F(\A')$ given by $\varphi(F) = F \cap K$.
  The image of $\varphi$ consists of the faces of $\A'$ that intersect the interior of $K$.

  By construction and by Lemma \ref{lemma:generic-point-subarrangement}, $x_0$ is still generic with respect to $\A'$ and all the chambers $D \in \C(\A)$ with $d(x_0, D) \leq d(x_0, C)$ intersect the interior of $K$.
  Thus, given a Euclidean order $\eu$ of $\C(\A)$ with base point $x_0$, there exists a Euclidean order $\euprime$ of $\C(\A')$ with base point $x_0$ such that $\varphi$ is an order-preserving bijection between the initial segment of $(\C(\A), \eu)$ up to $C$ and the initial segment of $(\C(\A'), \euprime)$ up to $\varphi(C)$.
    
  Consider the subcomplex $S = \eta^{-1}((\C\times\C)_{\leq (C,E)})$ of $\sal(\A)$.
  Since $\U \subseteq \F_K$, the map $\varphi$ induces an order-preserving and orientation-preserving injection $\psi \colon S \to \sal(\A')$ that maps a cell $\cell{D,G} \in S$ to the cell $\cell{\varphi(D),\varphi(G)} \in \sal(\A')$.
  Let $S' = \psi(S)$ be the copy of $S$ inside $\sal(\A')$.
  By definition of $S$, a fiber of $\eta$ is either disjoint from $S$ or entirely contained in $S$.
  Therefore, a non-critical cell of $S$ is matched with another cell of $S$.
  
  We now use the order $\euprime$ to construct a Euclidean matching $\M'$ on $\sal(\A')$.
  Denote by $\eta' \colon \sal(\A') \to \C(\A') \times \C(\A')$ the analogue of $\eta$ for the arrangement $\A'$ (see the proof of Theorem \ref{thm:matching}).
  Consider a fiber $\eta'^{-1}(C',E')$ that intersects $S'$.
  Then there is some cell $\cell{D',G'} \in \eta'^{-1}(C', E') \cap S'$, with $\cell{D',G'} = \psi(\cell{D,G})$ for some $\cell{D,G} \in S$.
  If we define $(\bar{C},\bar{E})=\eta(\cell{D,G})$, we have that $\varphi(\bar{C})=C'$ and $\varphi(\bar{E})=E'$, because by construction the cell $\cell{\varphi(D),\varphi(G)}$ is in the fiber $\eta^{-1}(\varphi(\bar{C}),\varphi(\bar{E}))$.
  By Corollary \ref{corollary:visible} and Lemma \ref{lemma:visible-faces}, the fiber $\eta'^{-1}(C',E')$ is isomorphic to the poset of faces of $E' \cap X_{C'}$ visible from some point $y_{C'}$ in the relative interior of $\C' \cap X_{C'}$.
  By construction of $\A'$, the map $\varphi$ induces a bijection between the faces of $E' \cap X_{C'}$ visible from $y_{C'}$ and the faces of $\bar E \cap X_{\bar C} = \bar E \cap X_{C'}$ visible from $y_{C'}$:
  if $F$ is a visible face of $\bar E \cap X_{C'}$, then $F \in \U$ and so $\varphi(F)$ is still visible;
  conversely, a visible face $F'$ of $E' \cap X_{\bar C}$ cannot be contained in any hyperplane of $\bar{\A}$, and by construction of $\A'$ it must also be a face of $\bar E$.
  Therefore the fiber $\eta'^{-1}(C',E')$ is the isomorphic image of the fiber $\eta^{-1}(\bar C, \bar E)$ under the map $\psi$.
  
  We have proved that a fiber of $\eta'$ is either disjoint from $S'$ or entirely contained in $S'$.
  Then we can choose the Euclidean matching $\M'$ so that its restriction to $S'$ coincides with the image of the restriction of $\M$ to $S$ under the isomorphism $\psi \colon S \to S'$.
  In particular, a cell $\cell{D,G} \in S$ is $\M$-critical if and only if $\psi(\cell{D,G}) \in S'$ is $\M'$-critical.
  
  Consider now a $\M$-critical cell $\cell{D,G} \in \sal(\A)$ such that there is at least one alternating path from $\cell{C,F_C}$ to $\cell{D,G}$.
  Since $\M$ is homogeneous with respect to $\eta$, every alternating path starting from $\cell{C,F}$ is entirely contained in $S$.
  In particular, $\cell{D,G} \in S$.
  Thus the map $\psi\colon S \to S'$ induces a bijection between the alternating paths from $\cell{C,F}$ to $\cell{D,G}$ in $\sal(\A)$ (with respect to the matching $\M$) and the alternating paths from $\psi(\cell{C,F})$ to $\psi(\cell{D,G})$ in $\sal(\A')$ (with respect to the matching $\M'$).
  In particular, the incidence number between $\cell{C,F}$ and $\cell{D,G}$ in the Morse complex $\sal(\A)_\M$ is the same as the incidence number between $\psi(\cell{C,F})$ and $\psi(\cell{D,G})$ in the Morse complex $\sal(\A')_{\M'}$.
  Since $\A'$ is finite, the Morse complex $\sal(\A')_{\M'}$ is minimal and all its incidence numbers vanish.
  Therefore the incidence number between $\cell{C,F}$ and $\cell{D,G}$ in $\sal(\A)_\M$ also vanishes.
\end{proof}

\tikzfading[name=fade outside,
inner color=transparent!0,
outer color=transparent!100]
\begin{figure}
   	\begin{tikzpicture}[label distance=-3]
    \clip (0,0) rectangle (12,6.2);
    
    \node[circle,inner sep=1.5pt, fill=black, label=left:{$x_0$}] (x0) at (8,3.5) {};
    \begin{scope}[every path/.style={draw=black, opacity=0.3, dashed}]
    \foreach \i in {1,...,6} {
    	\path (x0) circle (0.535*\i);
    }
    \end{scope}

   	\draw[name path=L3] (0,5) -- (12,0);
   	\fill[white] (4.1,3.25) circle(0.55);
    \draw[name path=L1] (0,1) -- (24,11);
    \draw[name path=L2] (9,0) -- (11,12);
    \draw[name path=L4] (0,3) -- (12,3);
    
    \path[name intersections={of=L1 and L2, by=x6}];
    \path[name intersections={of=L3 and L2, by=x7}];
    \path[name intersections={of=L1 and L3, by=x8}];
    \path[name intersections={of=L2 and L4, by=x4}];
    
    \node (C0) at (8.7,4) {$C_0=$\color{red}$\,F_0$};
    \node (C1) at (8.5,2.2) {$C_1$};
    \node (C2) at (6,5) {$C_2$};
    \node (C3) at (11,4) {$C_3$};
    \node (C4) at (11,1.8) {$C_4$};
    \node (C5) at (6,1) {$C_5$};
    \node (C6) at (10.8,5.9) {$C_6$};
    \node (C7) at (10.6,0.2) {$C_7$};
    \node (C8) at (1,3.8) {$C_8$};
    \node (C9) at (1,2.2) {$C_9$};

\begin{scope}[every node/.style={circle,inner sep=1.5pt,fill=red}]

      \node at (x4) {};
      \node at (x6) {};
      \node at (x7) {};
      \node at (x8) {};
    \end{scope}

\begin{scope}[every path/.style={draw=red, very thick, fill=red}]
     \draw (x6) -- (x8);
     \draw (x6)-- (x4);
\draw (x7) -- (x8);
     \draw (x4) -- (x8);
    \end{scope}
    
    \begin{scope}[every node/.style={text=red}]
	    \node (F1) at (7,2.75) {$F_1$};
	    \node (F2) at (7.1,4.23) {$F_2$};
	    \node (F3) at (9.9,4) {$F_3$};
	    \node (F4) at (9.7,2.7) {$F_4$};
	    \node (F5) at (7,1.8) {$F_5$};
	    \node (F6) at (10.13,5.48) {$F_6$};
	    \node (F7) at (9.35,0.85) {$F_7$};
	    \node (F8) at (4.17,3.25) {$F_8=F_9$};
    \end{scope}
    
    \fill[opacity=0.13, color=red] (x6) -- (x4) -- (x8);
  \end{tikzpicture}
  
  \caption{Euclidean order with respect to $x_0$.
  The faces $F_i = F_{C_i}$ defined in Theorem \ref{thm:euclidean-valid} are highlighted.}
  \label{fig:eucl-order}
\end{figure}
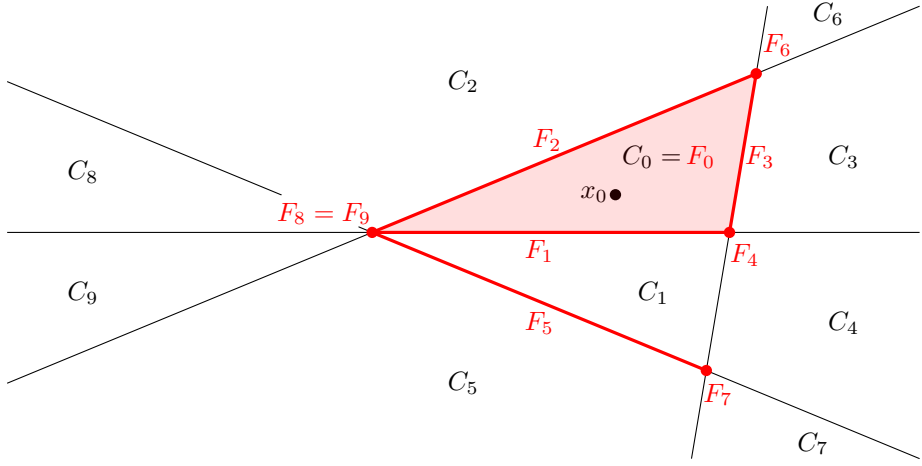

The following result is a direct consequence of Theorems \ref{thm:euclidean-matching} and \ref{thm:minimality}.
It gives a simple geometric way to compute the Betti numbers of the complement of an arrangement.

\begin{provedcorollary}[Betti numbers]
  \label{cor:betti-numbers}
  Let $\A$ be a (locally) finite hyperplane arrangement in $\R^n$, and let $x_0 \in \R^n$ be a generic point.
  The $k$-th Betti number of the complement $M(\A)$ is equal to the number of chambers $C$ such that the projection $\rho_C(x_0)$ lies in the relative interior of a face $F_C$ of codimension $k$.
  Equivalently, the Poincaré polynomial of $\A$ is given by
  \[ \pi(\A,t) = \sum_{C \in \,\C(\A)} t^{\,\codim F_C}. \qedhere \]
\end{provedcorollary}

In particular, Corollary \ref{cor:betti-numbers} solves a conjecture by Drton and Klivans on the characteristic polynomial of finite reflection arrangements \cite[Conjecture 6]{drton2010geometric}, since the coefficients of the characteristic and Poincaré polynomial coincide up to sign and reordering \cite[Definition 2.52]{orlik2013arrangements}.

\begin{example}
  Consider the line arrangement $\A$ of Figure \ref{fig:eucl-order}.
  For the given generic point $x_0$ in the interior of $C_0$, the computation of the Betti numbers $b_i$ according to Corollary \ref{cor:betti-numbers} goes as follows: there is one chamber (namely $C_0$) such that the projection of $x_0$ lies in its interior, so $b_0 = 1$; there are four chambers (namely $C_1$, $C_2$, $C_3$ and $C_5$) such that the projection of $x_0$ lies in the interior of a $1$-dimensional face, so $b_1 = 4$; finally, for the remaining chambers ($C_4$, $C_6$, $C_7$, $C_8$ and $C_9$) the projection of $x_0$ is a $0$-dimensional face, so $b_2 = 5$.
\end{example}

\begin{remark}
  For any choice of the generic point $x_0$, the only chamber that contributes to the $0$-th Betti number is the one containing $x_0$.
  In addition, for every hyperplane $H \in \A$ there is exactly one chamber $C$ such that $\rho_C(x_0) \in H$ and $\rho_C(x_0) \not\in H'$ for every $H' \in \A \setminus \{H\}$.
  Therefore Corollary \ref{cor:betti-numbers} immediately implies the well-known facts that $b_0(\A) = 1$ and $b_1(\A) = |\A|$.
\end{remark}

 \section{Brieskorn's Lemma and naturality}
\label{sec:brieskorn}

In this section we are going to relate the Morse complex of $\A$, constructed using a Euclidean matching, to the Morse complexes of subarrangements $\A_X$.

Given a flat $X \in \L(\A)$, for every face $\bar F \in \F(\A)$ such that $|\bar F|=X$ there is a natural inclusion of $\sal(\A_X)$ into $\sal(\A)$.
It maps a cell $\cell{D, G} \in \sal(\A_X)$ to the unique cell $\cell{C, F} \in \sal(\A)$ such that $\bar F \subseteq F \subseteq G$, $\dim F = \dim G$, and $C \subseteq D$.
We call this the inclusion of $\sal(\A_X)$ into $\sal(\A)$ \emph{around} $\bar F$.
Geometrically, this corresponds to including the complement of $\A_X^\mathbb{C}$, intersected with a neighborhood of some point in the interior of $\bar F$, into $M(\A)$.
The inclusions $\sal(\A_X) \hookrightarrow \sal(\A)$ that we are going to consider in this section are always of this type, for some face $\bar F \in \F(\A)$ with $|\bar F| = X$.

Recall from Definition \ref{def:euclidean-matching} that a Euclidean matching has a critical cell $\cell{C, F_C} \in \sal(\A)$ for every chamber $C$, where $F_C$ is the smallest face of $C$ containing $\rho_C(x_0)$.
Every critical cell $\cell{C, F_C}$ is thus associated to a flat $X_C = |F_C|$.
Conversely, given a flat $X \in \L(\A)$, the critical cells $\cell{C, F_C}$ associated to $X$ are exactly those for which $\rho_C(x_0) = \rho_X(x_0)$.

This simple observation yields a proof of Brieskorn's Lemma, a classical result in the theory of hyperplane arrangements due to Brieskorn \cite{brieskorn1973groupes}.
See also \cite[Lemma 5.91]{orlik2013arrangements} and \cite[Proposition 3.3.3]{callegaro2017integer}.

\begin{lemma}[Brieskorn's Lemma \cite{brieskorn1973groupes}]
  \label{lemma:brieskorn}
  Let $\A$ be a locally finite arrangement in $\R^n$.
  For every $k\geq 0$, there is an isomorphism
  \[ \theta_k \colon \bigoplus_{X \in \L_k} H_k(M(\A_X); \Z) \to H_k(M(\A); \Z) \]
  induced by suitable inclusions $j_X\colon \sal(\A_X) \hookrightarrow \sal(\A)$ of CW complexes.
  The inverse isomorphism $\theta_k^{-1}$ is induced by the natural inclusion maps $i_X \colon M(\A) \hookrightarrow M(\A_X)$.
\end{lemma}

\begin{proof}
  Let $x_0 \in \R^n$ be a generic point with respect to $\A$, and let $\M$ be a Euclidean matching on $\sal(\A)$ with base point $x_0$.
  Let $X \in \L_k$ be a flat of codimension $k$.
  By Lemma \ref{lemma:generic-point-subarrangement}, the point $x_0$ is generic also with respect to the subarrangement $\A_X$.
  Consider the inclusion $j_X \colon \sal(\A_X) \hookrightarrow \sal(\A)$ around the unique face of $\A$ containing the projection $\rho_X(x_0)$.
  Let $\M_X$ be a Euclidean matching on $\sal(\A_X)$ with base point $x_0$.
  
  All homology groups in this proof are with integer coefficients.
  By Theorem \ref{thm:minimality}, we have that $H_k(\sal(\A))$ is a free abelian group generated by elements of the form
  \[ [ \cell{C,F_C} + \text{a finite sum of non-critical $k$-cells}] \]
  for each critical $k$-cell $\cell{C,F_C}$ of $\sal(\A)$.
  Similarly, for every flat $X \in \L_k$, we have that $H_k(\sal(\A_X))$ is a free abelian group generated by elements of the same form as above, one for every critical $k$-cell of $\sal(\A_X)$.
  The critical $k$-cells of $\sal(\A_X)$ are in bijection (through the map $j_X$) with the critical $k$-cells $\cell{C,F}$ of $\sal(\A)$ such that $|F| = X$.
  Then the inclusions $j_X$ induce an isomorphism
  \[ \bar\theta_k \colon \bigoplus_{X \in \L_k} H_k(\sal(\A_X)) \to H_k(\sal(\A)). \]
  
  Let $\varphi\colon \sal(\A) \xhookrightarrow{\simeq} M(\A)$ and $\varphi_X\colon \sal(\A_X) \xhookrightarrow{\simeq} M(\A_X)$ be the homotopy equivalences constructed in \cite{salvetti1987topology}.
  Then the composition
  \[ \bigoplus_{X \in \L_k} H_k(M(\A_X)) \xrightarrow{\bigoplus (\varphi_X)_*^{-1}} \bigoplus_{X \in \L_k} H_k(\sal(\A_X)) \xrightarrow{\bar \theta_k} H_k(\sal(\A)) \xrightarrow{\varphi_*} H_k(M(\A)) \]
  is the isomorphism $\theta_k$ as in the statement.
  
  By naturality of Salvetti's construction, the following diagram is commutative up to homotopy.
  \begin{center}
  \begin{tikzcd}
    \sal(\A_X)
    \arrow[r, hook, "j_X"]
    \arrow[d, hook', "\varphi_X"]
    & \sal(\A) \arrow[d, hook', "\varphi"]
    \\
    M(\A_X)
    & M(\A) \arrow[l, hook', "i_X"]
  \end{tikzcd}
  \end{center}
  Looking at the induced commutative diagram in homology, we obtain that the inverse isomorphism $\theta_k^{-1}$ is induced by the inclusion maps $i_X$.
\end{proof}

If we fix a flat $X \in \L(\A)$, it is possible to choose the base point $x_0$ so that the Morse complex of $\sal(\A_X)$ injects into the Morse complex of $\sal(\A)$.
We prove this naturality property in the following lemma.

\begin{lemma}
  Let $X \in \L(\A)$ be a flat, and fix an inclusion $j \colon \sal(\A_X) \hookrightarrow \sal(\A)$ around some face $\bar F$ with $|\bar F| = X$.
  There exist Euclidean matchings $\M_X$ and $\M$, on $\sal(\A_X)$ and $\sal(\A)$ respectively, such that:
  \begin{enumerate}[(i)]
    \item they share the same base point $x_0$;
    \item $(j \times j)(\M_X) \subseteq \M$;
    \item the inclusion $j$ induces an inclusion of the Morse complex of $\sal(\A_X)$ into the Morse complex of $\sal(\A)$.
  \end{enumerate}
\end{lemma}

\begin{proof}
  Let $x_0 \in \R^n$ be a generic point such that $d(x_0, \bar F) < d(x_0, H)$ for every hyperplane $H \in \A \setminus \A_X$ (the existence of $x_0$ follows from Lemma \ref{lemma:eucpt}).
  For example, we can choose a point $y$ in the relative interior of $\bar F$, and then take $x_0$ in a small neighborhood of $y$.
  
  Let $\eu$ and $\euprime$ be Euclidean orders with base point $x_0$ on $\C(\A)$ and $\C(\A_X)$, respectively.
  Notice that, by construction of $x_0$, the total order $\eu$ starts with the chambers containing $\bar F$.
  
  Let $\eta\colon \sal(\A) \to \C(\A) \times \C(\A)$ be the poset map defined in the proof of Theorem \ref{thm:matching}, induced by the total order $\eu$.
  Let $\eta'\colon \sal(\A_X) \to \C(\A_X) \times \C(\A_X)$ be the analogous poset map for the arrangement $\A_X$, induced by the total order $\euprime$.
  Then, for every pair of chambers $C,E \in \C(\A)$ containing $\bar F$, we have
  \[ \eta^{-1}(C,E) = j( {\eta'}^{-1} (\pi_X(C), \pi_X(E))). \]
  In other words, the inclusion $j$ maps fibers of $\eta'$ to fibers of $\eta$.
  Notice that, by Remark \ref{rmk:matching-independent}, these fibers only depend on $x_0$ and not on the particular choices of the Euclidean orders $\eu$ and $\euprime$.
  
  Let $\M_X$ be a Euclidean matching on $\sal(\A_X)$ with base point $x_0$.
  Recall that such a matching is constructed on the fibers of $\eta$ (see Definition \ref{def:euclidean-matching}).
  Then there exists a Euclidean matching $\M$ on $\sal(\A)$ with base point $x_0$ that contains $(j \times j)(\M_X)$.
  
  The alternating paths in $\sal(\A)$ starting from cells in the subcomplex $j(\sal(\A_X))$ remain in this subcomplex.
  Therefore $j$ induces an inclusion of the Morse complex of $\sal(\A_X)$ (with respect to the matching $\M_X$) into the Morse complex of $\sal(\A)$ (with respect to the matching $\M$).
\end{proof}

 \section{Local system homology of line arrangements}
\label{sec:homology}

In the case of a line arrangement in $\R^2$ it is possible to explicitly describe alternating paths between critical cells of a Euclidean matching.
As an application, in this section we are going to describe the algebraic Morse complex that computes the homology of the complement $M(\A)$ with coefficients in an abelian local system.
Then we are going to compare the obtained complex with the algebraic complex of Gaiffi and Salvetti \cite{gaiffi2009morse}, which is based on the polar matching of Salvetti and Settepanella \cite{salvetti2007combinatorial}.

Let $\A$ be a locally finite line arrangement in $\R^2$.
An abelian local system $L$ on $M(\A)$ is determined by the elements $t_\ell \in \Aut(L)$ associated to elementary positive loops around every line $\ell \in \A$ (cf.\ \cite[Section 2.4]{gaiffi2009morse}).
The boundaries $\partial_i$ of the algebraic Morse complex are determined by the incidence numbers $[\cell{D,G}, \cell{C,F}]^\M \in \Z[t_\ell^{\pm 1}]_{\ell \in \A}$, between critical $i$-cells $\cell{D,G}$ and critical $(i-1)$-cells $\cell{C,F}$, in the Morse complex.

We refer to \cite[Section 5]{salvetti2007combinatorial} for a detailed explanation of how to compute these incidence numbers, given an acyclic matching on the Salvetti complex $\sal(\A)$.
We only make the following substantial change of convention with respect to \cite{salvetti2007combinatorial, gaiffi2009morse}: given a cell $\cell{C,F}$, we choose as its representative point the $0$-cell $\cell{C^F,C^F}$, where $C^F$ is the chamber opposite to $C$ with respect to $F$ (the role of the representative point is thoroughly described in \cite[Section 9]{steenrod1943homology}).
It is more convenient to choose $\cell{C^F,C^F}$ instead of $\cell{C,C}$, because in this way two matched cells have the same representative point.

We recall some useful definitions and facts from \cite[Chapter 5]{salvetti2007combinatorial}, adapting them to our different convention on the representative point.
Given two chambers $D$ and $C$, denote by $u(D,C)$ a combinatorial positive path of minimal length from $\cell{D,D}$ to $\cell{C,C}$, in the $1$-skeleton of $\sal(\A)$.
In particular, let $\Gamma(C) = u(C, C_0)$ be a minimal positive path from the chamber $C$ to a base chamber $C_0$.
Every path $u(D,C)$ crosses each line at most once by \cite[Lemma 5.1]{salvetti2007combinatorial}.
Consider the closed path $\Gamma(D)^{-1} \, u(D,C) \, \Gamma(C)$, which starts from $C_0$, passes through $D$ and $C$, and then goes back to $C_0$.
This path determines an element $\bar u(D,C) \in H_1(M(\A))$ which is equal to the product of the positive loops around the lines in $s(C_0,C) \cap s(D,C)$.
Then the incidence number $[\cell{D,G}, \cell{C,F}] \in \Z[t_\ell^{\pm 1}]_{\ell \in \A}$ between an $i$-cell $\cell{D,G}$ and an $(i-1)$-cell $\cell{C,F}$ in $\sal(\A)$ is given by
\begin{equation*}
  [\cell{D,G}: \cell{C,F}]= [\cell{D,G}: \cell{C,F}]_{\Z} \; \bar{u}(D^G,C^F),
\end{equation*}
where $[\cell{D,G} : \cell{C,F}]_{\Z} = \pm 1$ denotes the integral incidence number in $\sal(\A)$.

Let $x_0 \in \R^2$ be a generic point with respect to the line arrangement $\A$, and fix a Euclidean matching $\M$ on the Salvetti complex $\sal(\A)$ with base point $x_0$.
Let $C_0$ be the chamber containing $x_0$ (this is the first chamber in any Euclidean order with base point $x_0$).
Recall that the matching $\M$ is constructed on the fibers of the map $\eta \colon \sal(\A) \to \C \times \C$.

To compute the algebraic Morse complex (see \cite[Definition 11.23]{kozlov2007combinatorial}), we first need to describe the alternating paths between critical cells.
The alternating paths between a critical $1$-cell $\cell{C,F}$ and the only critical $0$-cell $\cell{C_0,C_0}$ are particularly simple, since all the $0$-cells are in $N(C_0)$.

\begin{lemma}\label{Lemma:1path}
  Let $\cell{C,F}$ be a critical $1$-cell.
  Denote by $C'$ the unique chamber containing $F$ other than $C$.
  There are exactly two alternating paths from $\cell{C,F}$ to the only critical $0$-cell $\cell{C_0,C_0}$:
  \begin{itemize}
    \item $\cell{C,F} \searrow \cell{C,C} \nearrow \cell{C',F} \searrow \cell{C',C'} \nearrow \dots \searrow \cell{C_0,C_0}$
    \item $\cell{C,F} \searrow \cell{C',C'} \nearrow \dots \searrow \cell{C_0,C_0}$
  \end{itemize}
  (after $\cell{C',C'}$, they continue in the same way).
\end{lemma}

\begin{proof}
  Since $\cell{C,F}$ is critical, the line $|F|$ separates $C$ and $C_0$.
  In the boundary of the $1$-cell $\cell{C,F}$ there are the two $0$-cells $\cell{C,C}$ and $\cell{C',C'}$.
  The $0$-cell $\cell{C,C}$ is matched with the $1$-cell $\cell{C',F}$, because these are the unique cells in the fiber $\eta^{-1}(C_0, C)$.
  Then an alternating path starting with $\cell{C,F} \searrow \cell{C,C}$ is forced to continue with $\nearrow \cell{C',F} \searrow \cell{C',C'}$.
  After $\cell{C',C'}$ there is exactly one way to continue the path, because every non-critical $0$-cell is matched with some $1$-cell, and this $1$-cell has exactly one other $0$-cell in the boundary.
  Since the matching is proper, one such path must eventually reach the critical $0$-cell $\cell{C_0,C_0}$.
\end{proof}

We can use the previous lemma to compute the boundary $\partial_1$.
The resulting formula coincides with the one of \cite[Proposition 4.1]{gaiffi2009morse}.

\begin{proposition}
  The incidence number between a critical $1$-cell $\cell{C,F}$ and the only critical $0$-cell $\cell{C_0,C_0}$ in the Morse complex is given by
  \[ [\cell{C,F} : \cell{C_0, C_0}]^\M = (1-t_{\,|F|}). \]
\end{proposition}

\begin{proof}
The orientation of a $1$-cell $\cell{\tilde{C},\tilde{F}}$ is defined so that $[\cell{\tilde{C}, \tilde{F}}, \cell{\tilde C^{\tilde F}, \tilde C^{\tilde F}}]_\Z = 1$.
Now, if $\cell{\tilde{C},\tilde{F}} \in N(C_0)$, then $\tilde{C}$ is closer to $C_0$ with respect to $\tilde{C}^{\tilde{F}}$ and so we have that:
\begin{equation*}
 [\cell{\tilde{C},\tilde{F}},\cell{\tilde{C},\tilde{C}}] = -1; \qquad [\cell{\tilde{C},\tilde{F}}:\cell{\tilde{C}^{\tilde{F}},\tilde{C}^{\tilde{F}}}] = 1.
\end{equation*}

By Lemma \ref{Lemma:1path} we see that there are exactly two alternating paths between $\cell{C,F}$ and $\cell{C_0,C_0}$, and by  \cite[Definition 11.23]{kozlov2007combinatorial} the incidence number in the Morse complex is given by
\begin{equation*}
  [\cell{C,F} : \cell{C_0, C_0}]^\M = [\cell{C,F}:\cell{C,C}]+[\cell{C,F}:\cell{C',C'}].
\end{equation*}
Since $|F| \in s(C_0,C) \cap s(C^F,C)$, the first term is
\begin{equation*}
  [\cell{C,F}:\cell{C,C}]=[\cell{C,F}:\cell{C,C}]_{\Z} \bar{u}(C^F,C)=-t_{|F|},
\end{equation*}
The second term is given by
\begin{equation*}
  [\cell{C,F}:\cell{C',C'}]=[\cell{C,F}:\cell{C',C'}]_{\Z} \bar{u}(C^F,C')=\bar{u}(C',C')=1. \qedhere
\end{equation*}                                                                                                                                                                                                   
\end{proof}

Now we want to compute the boundary $\partial_2$.
To simplify the notation, denote a $2$-cell $\cell{D,\{p\}}$ also by $\cell{D,p}$, where $p \in \R^2$ is the intersection point of two or more lines of $\A$.

It is convenient to assign the orientation of the $2$-cells so that they behave well with respect to the matching.
Given a $2$-cell  $\cell{D,p} \not\in N(C_0)$,  we choose the orientation in the following way.
Let $\ell, \ell'$ be the two walls of $D$ that pass through $p$.
Let $\ell$ be the one that does not separate $D$ from $C_0$ if it exists, or otherwise the closest one to $x_0$.
Then the orientation of $\cell{D,p}$ is the one for which $[\cell{D,p}, \cell{D,\ell}]_\Z = 1$.
The orientation of the $2$-cells in $N(C_0)$ is assigned arbitrarily. 
The reason of this choice is that the incidence number between two matched cells is always $+1$.
Indeed, if $C'$ is the chamber such that $X_{C'}=\ell'$, then $\cell{D,p} \in N(C')$ by construction.

We are going to show that there is a correspondence between alternating paths from critical $2$-cells to critical $1$-cells and certain sequences of elements of $\L_1(\A)$.
Consider an alternating path of the form
\begin{equation}\label{eq:AltPath}
  \cell{D,p} \searrow \cell{C_1,F_1} \nearrow \cell{D_1,p_1} \searrow \cell{C_2,F_2} \nearrow \dots \searrow \cell{C_n,F_n},
\end{equation}
where $\cell{D,p}$ is a critical $2$-cell and $\cell{C_n,F_n}$ is a critical $1$-cell.
By construction of the matching, none of the cells in \eqref{eq:AltPath} belongs to $N(C_0)$.
We have that the starting cell $\cell{D,p}$ and the sequence $(F_1,\ldots,F_n)$ uniquely determine the alternating path.
This is because for each $i$ there are only two cells of the form $\cell{C',F_i}$ for some $C' \in \C$, and one of these cells is in $N(C_0)$.
Thus $C_i$ is uniquely determined by $F_i$ for every $i$, and $\cell{D_i, p_i}$ is the cell matched with $\cell{C_i, F_i}$.

We are now going to describe which sequences in $\L_1(\A)$ give rise to an alternating path.
Given a face $F \in \L_1(\A)$, let $\ell = |F|$.
If $\rho_\ell(x_0)\notin F$, we denote by $p(F)$ the endpoint of $F$ which is closer to $\rho_\ell(x_0)$.
In addition, let $C(F)$ be the unique chamber containing $F$ such that $\cell{C(F),F} \notin N(C_0)$.

\begin{definition}
  \label{Def:valid-sequence}
  Given two different faces $F,\, G \in \L_1(\A)$, we write $F \to G$ if
  \begin{itemize}
    \item $F \cap G = \{p(F)\}$;
    \item $|F|=|G|$, or $F$ and $C_0$ lie in the same half-plane with respect to $|G|$. 
  \end{itemize}
\end{definition}

\begin{lemma}\label{Lemma:sequence}
  Let $\cell{D,p}$ be a critical $2$-cell and $\cell{C,F}$ a critical $1$-cell.
  The alternating paths between $\cell{D,p}$ and $\cell{C,F}$ are in one to one correspondence with the sequences in $\L_1(\A)$ of the form $(F_1 \to F_2 \to \ldots \to F_n=F)$ such that $\cell{C(F_1),F_1} < \cell{D,p}$.
\end{lemma}

\begin{proof}
  Consider an alternating path as in \eqref{eq:AltPath}.
  We have already seen that such a path is completely determined by the starting cell $\cell{D,p}$ and by the sequence $(F_1,\ldots,F_n)$.
  Clearly the condition $\cell{C(F_1),F_1} < \cell{D,p}$ must be satisfied.
  We want to show that $F_i \to F_{i+1}$ for each $i=1,\dots,n-1$. 
 
  Let $E_i$ be the chamber opposite to $C(F_i)$ with respect to $F_i$.
  By construction of the matching, it is immediate to see that the cell $\cell{C(F_i),F_i}$ is matched with $\cell{D(F_i),p(F_i)}$, where $D(F_i)$ is the chamber opposite to $E_i$ with respect to $p(F_i)$.
  By hypothesis we have that $\cell{C(F_{i+1}),F_{i+1}} < \cell{D(F_i),p(F_i)}$ which implies that $F_i \cap F_{i+1} = \{p(F_i)\}$ and that $D(F_i).F_{i+1}=C(F_{i+1})$.
  Since $\cell{C(F_{i+1}),F_{i+1}} \notin N(C_0)$, we have that $C_0$ and $C(F_{i+1})$ are in opposite half-planes with respect to $|F_{i+1}|$.
  The same is true for $F_i$ and $C(F_{i+1})$, because $D(F_i)$ and $F_i$ are in opposite half-planes with respect to $|F_{i+1}|$, unless $F_i \subset |F_{i+1}|$.
  Then we have that $F_i \to F_{i+1}$.
 
  Conversely, we now prove that every sequence $(F_1 \to F_2 \to \dots \to F_n=F)$ satisfying $\cell{C(F_1), F_1} < \cell{D,p}$ has an associated alternating path.
  We do this by induction on the length $n$ of the sequence.
  
  The case $n=1$ is trivial, since we already know that $\cell{C(F_1),F_1}<\cell{D,p}$.
  In the induction step, we need only to prove that $F \to G$ implies $\cell{C(G),G} < \cell{D(F),p(F)}$.
  From the first condition of Definition \ref{Def:valid-sequence}, we have that $G \prec \{p(F)\}$.
  We need to check that $D(F).G=C(G)$.
  By definition of $C(G)$, this is equivalent to proving that $D(F)$ and $C_0$ lie in opposite half-planes with respect to $|G|$.
  This is true because $F$ and $C_0$ lie in the same half-plane with respect to $|G|$.
\end{proof}

Now that we have a description of the alternating paths, we can use it to compute the boundary of the Morse complex.

\begin{definition}
  Given two different faces $F, \,G \in \L_1(\A)$, let
  \begin{equation*}
    [F \to G]= \frac{[\cell{D(F),p(F)}:\cell{C(G),G}]}{[\cell{D(F),p(F)}:\cell{C(F),F}]},                                      
  \end{equation*}
  where the incidence numbers on the right are taken in the Salvetti complex $\sal(\A)$, and $D(F)$ is defined as in the proof of Lemma \ref{Lemma:sequence}.
\end{definition}

\begin{lemma} \label{Lemma:boundary-line}
  Given two different faces $F,G \in \F_1(\A)$ such that $F \to G$, we have
  \begin{equation*}
  [F \to G]=\pm \prod t_\ell,
  \end{equation*}
  where the product is on the set of lines $\ell \neq |G|$ passing through $p(F)$, such that $G$ and $C_0$ lie in opposite half-planes, whereas $F$ and $C_0$ lie in the same closed half-plane (with respect to $\ell$). The sign is $+1$ if $p(F)=p(G)$, and $-1$ otherwise.
\end{lemma}

\begin{proof}
  Denote by $E(F)$ and $E(G)$ the chambers $C(F)^F$ and $C(G)^G$, respectively.
  Notice that $E(F)=D(F)^{p(F)}$, and therefore $[\cell{D(F),p(F)}:\cell{C(F),F}]= 1$.
  See Figure \ref{fig:F-to-G} for an example.
  
  Now we need to determine $\bar{u}(E(F),E(G))$, which is the product of the positive loops around the lines in $s(C_0,E(G)) \cap s(E(F),E(G))$.
  By definition of $E(G)$, we have that $s(C_0,E(G))$ is the set of lines different from $|G|$ for which $G$ and $C_0$ in opposite half-planes.
  Since every line in $s(E(F),E(G))$ goes through $p(F)$, it is now easy to see that $s(C_0,E(G)) \cap s(E(F),E(G))$ is the set described in the statement.
  
  We now need to determine the sign. If $p(G)=p(F)$, then we immediately see that $G$ is in the half-plane delimited by $|F|$ that contains $D(F)$.
  The opposite is true if $p(G) \neq p(F)$.
  By our choice of the orientation, we obtain the stated result.
\end{proof}

\begin{figure}
  \begin{center}
  \begin{tikzpicture}[label distance=-3]
    \draw[name path=L1] (-5,3) -- (7,3);
    \draw[name path=L2] (-5,0) -- (7,0);
    \draw[name path=L3] (-4,-2) -- (7,3.5);
    \draw[name path=L4] (-2,-3) -- (3,4.5);
    \draw[name path=L5] (3,-3) -- (-4.2,4.2);
    
    \node[circle,inner sep=1.5pt, fill=black, label=left:{$x_0$}] (x0) at (-3,0.5) {};
    
    \node (C0) at (-4,1.5) {$C_0$};
    \node (CF) at (2.3,1.9) {$C(F)$};
    \node (EF) at (-0.4,1.9) {$E(F)$};
    \node (CG) at (-2,-1.8) {$C(G)$};
    \node (EG) at (-4,-0.8) {$E(G)$};
    
    \path[name intersections={of=L1 and L4, by=a}];
    
    \begin{scope}[every path/.style={draw=blue, very thick, fill=red}]
      \draw (a) -- (0,0);
      \draw (0,0) -- (-4,-2);
    \end{scope}
    
    \begin{scope}[every node/.style={text=blue}]
      \node (F) at (0.8,1.6) {$F$};
      \node (G) at (-2.7,-1.05) {$G$};
    \end{scope}

    \fill[white] (0.3,-0.31) circle(0.25);
    \node[circle,inner sep=1.5pt, fill=black, label=below right:{$\!p(F)$}] (p) at (0,0) {};

  \end{tikzpicture}
  \end{center}
  
  \caption{Faces $F,G \in \F_1(\A)$ such that $F \to G$, as in Lemma \ref{Lemma:boundary-line}.}
  \label{fig:F-to-G}
\end{figure}
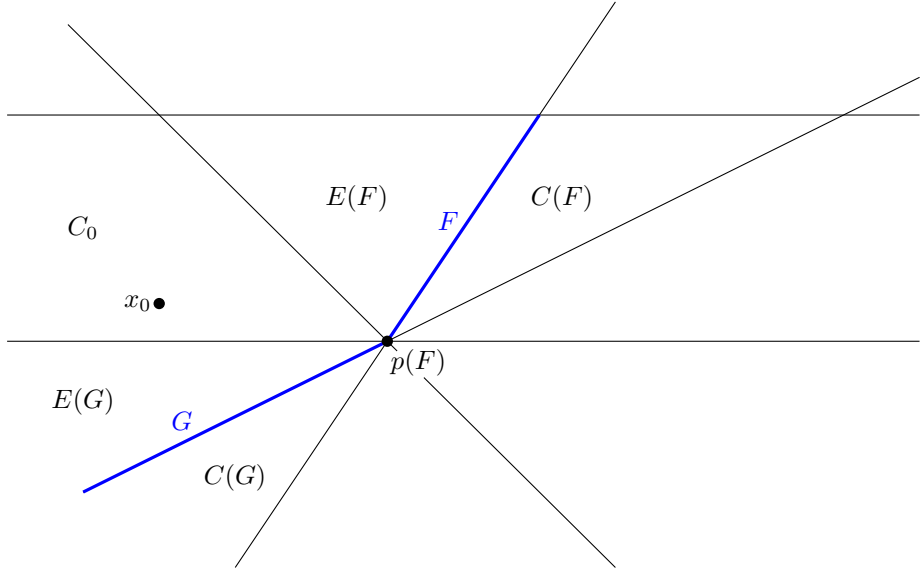

\begin{theorem}\label{Thm:line-complex}
  Let $\A$ be a locally finite line arrangement in $\R^2$.
  Let $\cell{D,p}$ be a critical $2$-cell and $\cell{C,F}$ a critical $1$-cell.
  Then their incidence number in the Morse complex is given by
  \begin{equation*}
    [\cell{D,p}:\cell{C,F}]^\M =\sum_{s \in \Seq}{\omega(s)},
  \end{equation*}
  where $\Seq$ is the set of sequences of Lemma \ref{Lemma:sequence}, and for each sequence $s=(F_1 \to F_2 \to \dots \to F_n=F) \in \Seq$ we define
  \begin{equation*}
    \omega(s)=(-1)^n \, [\cell{D,p}:\cell{C(F_1),F_1}] \, \prod_{i=1}^{n-1}{[F_i \to F_{i+1}]}.
  \end{equation*}
\end{theorem}

\begin{proof}
 It follows directly from \cite[Definition 11.23]{kozlov2007combinatorial} and Lemma \ref{Lemma:sequence}.
\end{proof}

\begin{remark}
 
Computing the incidence numbers is not the only way to obtain the  local system homology of a line arrangement. For example a different, more algebraic approach can be found in \cite{yoshinaga2014resonant}.
\end{remark}

\begin{example}[Deconing $A_3$]
  Consider the line arrangement $\A$ of Figure \ref{fig:dec-A3}, obtained by deconing the reflection arrangement of type $A_3$.
  Given a chamber $C_i$, denote by $\cell{C_i,F_i}$ the associated critical cell if it is of dimension $1$, or by $\cell{C_i,p_i}$ if it is of dimension $2$.
  Applying Theorem \ref{Thm:line-complex} and Lemma \ref{Lemma:boundary-line}, we obtain the boundary matrix $\partial_2$ of Table \ref{tab:decA3-2}.
  This matrix is slightly simpler than the one computed in \cite[Section 7]{gaiffi2009morse}, but there are many similarities.
  Specializing to the case $t_1=\ldots=t_5=t$, we obtain that
  \begin{equation*}
    H_1(M(\A); \, \Q[t^{\pm 1}]) \cong \left( \frac{\Q[t^{\pm 1}]}{t-1} \right)^3 \oplus \frac{\Q[t^{\pm 1}]}{t^3-1},
  \end{equation*}
  as already computed for example in \cite{gaiffi2009morse}.
\end{example}

\begin{figure}
  \begin{tikzpicture}[label distance=-3]
  
   \draw[name path=L1] (0,6.5) -- (6.5,0);
   \draw[name path=L2] (10,6.5) -- (3.5,0);
   \draw[name path=L3] (0,4) -- (10,4);
   \draw[name path=L4] (10,1.5) -- (3.5,8);
   \draw[name path=L5] (0,1.5) -- (6.5,8);
    
    \path[name intersections={of=L1 and L2, by=x6}];
    \path[name intersections={of=L3 and L2, by=x7}];
    \path[name intersections={of=L1 and L3, by=x8}];
    \path[name intersections={of=L2 and L4, by=x4}];
    
    \node[circle,inner sep=1.5pt, fill=black, label=left:{$x_0$}] (x0) at (6,3.6) {};

    \node (C0) at (5,2.9) {$C_0$};
    \node (C1) at (5,5.1) {$C_1$};
    \node (C2) at (7.5,2) {$C_2$};
    \node (C3) at (7.5,6) {$C_3$};
    \node (C4) at (9,4.6) {$C_4$};
    \node (C5) at (9,3.4) {$C_5$};
    \node (C6) at (2.5,2) {$C_6$};
    \node (C7) at (5,0.5) {$C_7$};
    \node (C8) at (2.5,6) {$C_8$};
    \node (C9) at (5,7.5) {$C_9$};
    \node (C10) at (1,4.6) {$C_{10}$};
    \node (C11) at (1,3.4) {$C_{11}$};
    
    \node (l1) at (0.5,6.5) {$l_1$};
    \node (l2) at (9.5,1.5) {$l_2$};
    \node (l3) at (9.7,3.7) {$l_3$};
    \node (l4) at (9.5,6.5) {$l_4$};
    \node (l5) at (0.5,1.5) {$l_5$};

  \end{tikzpicture}

  \caption{Deconing $A_3$.}
  \label{fig:dec-A3}
\end{figure}
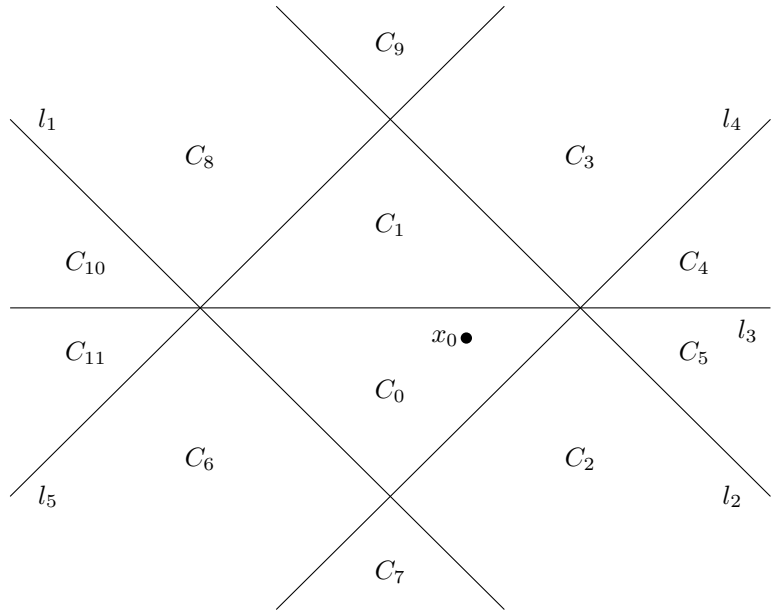

\begin{table}[htbp]
  \begin{tabular}{l|cccccc|}
    & $\cell{C_4,p_4}$ & $\cell{C_5,p_5}$ & $\cell{C_7,p_7}$ & $\cell{C_9,p_9}$ & $\cell{C_{10},p_{10}}$ & $\cell{C_{11},p_{11}}$ \\
  &&&&&&\\[-1em]
  \hline 
  &&&&&&\\[-1em]
  $\cell{C_1,F_1}$ & $1-t_4$ & $t_4(t_2-1)$ & $0$ & $0$ & $t_1-1$ & $t_1(1-t_5)$ \\
  
  $\cell{C_2,F_2}$ & $t_2t_3-1$ & $t_2-1$ & $1-t_1$ & $0$ & $0$ & $0$ \\
  
  $\cell{C_3,F_3}$ & $t_3(1-t_4)$ & $1-t_3t_4$ & $0$ & $t_5-1$ & $0$ & $0$ \\
  
  $\cell{C_6,F_6}$ & $0$ & $0$ & $t_4-1$ & $0$ & $1-t_3t_5$ & $1-t_5$\\
  
  $\cell{C_8,F_8}$ & $0$ & $0$ & $0$ & $1-t_2$ & $t_3(t_1-1)$ & $t_1t_3-1$ \\
  \hline
  \end{tabular}
  \bigskip
  
  \caption{The boundary $\partial_2$ of the deconing of $A_3$.}
  \label{tab:decA3-2}
\end{table}

\section*{Acknowledgements}

We would like to thank our supervisor, Mario Salvetti, for always giving valuable advice.
We also thank Emanuele Delucchi, Simona Settepanella, Federico Glaudo, and Daniele Semola, for the useful discussions.

\bibliographystyle{amsalpha-abbr}
\bibliography{bibliography}

\end{document}